\documentclass[11pt]{amsart}

\usepackage[text={450pt,660pt},centering]{geometry}

\usepackage{color}
\usepackage{esint,amssymb}
\usepackage{graphicx}
\usepackage{MnSymbol}
\usepackage{amsmath, mathtools}
\usepackage[colorlinks=true, pdfstartview=FitV, linkcolor=blue, citecolor=blue, urlcolor=blue,pagebackref=false]{hyperref}
\usepackage{microtype}

\usepackage{bm}
\usepackage{scalerel} %for scaling the boxes
\usepackage{dsfont}
\usepackage{mathrsfs}
\usepackage[font={footnotesize}]{caption}
\usepackage{stmaryrd}
\usepackage{tikz}

\usepackage{enumitem}

\definecolor{darkgreen}{rgb}{0,0.5,0}
\definecolor{darkblue}{rgb}{0,0,0.7}
\definecolor{darkred}{rgb}{0.9,0.1,0.1}

\makeatletter
\newtheorem*{rep@theorem}{\rep@title}
\newcommand{\newreptheorem}[2]{%
\newenvironment{rep#1}[1]{%
 \def\rep@title{#2 \ref{##1}}%
 \begin{rep@theorem}}%
 {\end{rep@theorem}}}
\makeatother

\newtheorem{theorem}{Theorem}
\newreptheorem{theorem}{Theorem}
\newreptheorem{proposition}{Proposition}
\newreptheorem{lemma}{Lemma}
\newtheorem{proposition}{Proposition}
\newtheorem{lemma}[proposition]{Lemma}

\theoremstyle{remark}

\theoremstyle{definition}
\newtheorem{definition}[proposition]{Definition}
\newtheorem{remark}[proposition]{Remark}

\numberwithin{equation}{section}
\numberwithin{proposition}{section}
\newcommand{\Z}{\mathbb{Z}}
\newcommand{\N}{\mathbb{N}}
\newcommand{\R}{\mathbb{R}}

\newcommand{\E}{\mathbb{E}}
\renewcommand{\P}{\mathbb{P}}

\newcommand{\Zd}{\mathbb{Z}^d}
\newcommand{\Rd}{{\mathbb{R}^d}}

\newcommand{\di}{\mathrm{d}}
\newcommand{\inte}[1]{%
 {\kern0pt#1}^{\mathrm{o}}%
}

\renewcommand{\a}{\mathbf{a}}
\renewcommand{\subset}{\subseteq}

\DeclareMathOperator*{\osc}{osc}

\newcommand{\indc}{\mathds{1}}

\newcommand{\X}{\mathcal{X}}

\begin{document}

\title[Convergence to the thermodynamic limit]{Convergence to the thermodynamic limit for random-field random surfaces}

\author[P. Dario]{Paul Dario}
\address[P. Dario]{Universit\'e Claude Bernard Lyon 1, Institut Camille Jordan, 69622 Villeurbanne, France}
\email{paul.dario@univ-lyon1.fr}

\maketitle

\begin{abstract}
We study random surfaces with a uniformly convex gradient interaction in the presence of quenched disorder taking the form of a random independent external field. Previous work on the model has focused on proving existence and uniqueness of infinite-volume gradient Gibbs measures with a given tilt and on studying the fluctuations of the surface and its discrete gradient.

In this work we focus on the convergence of the thermodynamic limit, establishing convergence of the finite-volume distributions with Dirichlet boundary conditions to translation-covariant (gradient) Gibbs measures. Specifically, it is shown that, when the law of the random field has finite second moment and is symmetric, the distribution of the gradient of the surface converges in dimensions $d\geq4$ while the distribution of the surface itself converges in dimensions $d\geq 5$. Moreover, a power-law upper bound on the rate of convergence in Wasserstein distance is obtained. The results partially answer a question discussed by Cotar and K\"{u}lske~\cite{CK12}.

\end{abstract}

\setcounter{tocdepth}{1}
\tableofcontents

\section{Introduction}
\subsection{Description of the model and main results} \label{section1.11} Random surfaces in statistical mechanics are used to model the interface separating two pure thermodynamic phases. In classical effective interface models of this type, the interface is represented by a real-valued field $\left\{ \phi(x) \, : \, x \in \Zd \right\}$ to which one associates an energy determined by an interaction potential applied to the discrete gradient of the surface. In this article, we consider a model of random interfaces in the presence of a random disorder taking the form of a random, independent field. Let us define this model precisely. Given finite set $\Lambda \subset \Z^d$, a map $\eta : \Lambda \to \R$ and an intensity parameter $\lambda > 0$, we assign to each surface $\phi:\Lambda \to \R$ satisfying the Dirichlet boundary condition $\phi = 0$ on $\partial\Lambda$ the energy
 \begin{equation} \label{eq:101071608}
    H_\Lambda^\eta \left( \phi \right) := \sum_{\substack{x, y \in \Lambda^+ \\ |x - y| = 1}} V(\phi(x) - \phi(y)) - \lambda \sum_{x \in \Lambda} \eta(x) \phi(x).
\end{equation}
where $\partial\Lambda$ is the external vertex boundary of $\Lambda$, $\Lambda^+$ is the set $\Lambda\cup\partial\Lambda$, $|\cdot|$ denotes the Euclidean norm and $V:\R \to \R$ is an interaction potential satisfying suitable assumptions. Typically, one considers even $V$ which are twice-continuously differentiable and uniformly convex, i.e., there exist two constants $c_-$ and $c_+$ satisfying
\begin{equation}\label{eq:V ellipticity}
  0 < c_- \leq V''(t) \leq c_+ < \infty.
\end{equation}
The law of random surface is then given by
\begin{equation} \label{eq:defmuLeta}
    \mu_\Lambda^\eta (\di \phi) := \frac{1}{Z_\Lambda^\eta}\exp \left( - H_{\Lambda}^{\eta} (\phi) \right) \prod_{x \in \Lambda} d \phi(x),
\end{equation}
where $d \phi(x)$ denotes the Lebesgue measure on $\R$ and $Z_\Lambda^\eta$ is the constant ensuring that $\mu_\Lambda^\eta$ is a probability measure. We assume that the external field $\eta$ is random, with the random variables $\left( \eta(x) \right)_{x \in \Zd}$ independent with symmetric distribution and unit variance, and wish to study the properties of the random surface distributed according to $\mu_{\Lambda}^\eta$ for a typical realization of the random field $\eta$. We denote by $\P$ the law of the random field, and by $\E$ the corresponding expectation.

A natural question pertaining to the model is whether the distribution~\eqref{eq:defmuLeta} of the random surface or the one of its discrete gradient converges upon considering an increasing sequence of domains $\Lambda_n \subseteq \Zd$ tending to $\Zd$ as $n$ tends to infinity. In the case of the pure system (i.e., when $\eta \equiv 0$), Funaki and Spohn~\cite{FS} proved the existence of this limit for the discrete gradient of the random surface in any dimension $d \geq 1$, as well as the uniqueness of the infinite-volume ergodic gradient Gibbs measures with a specific tilt. Regarding the distribution of the surface, it is known that the limit does not exist in dimensions $d \leq 2$~\cite{BLL75} and, by using the Brascamp-Lieb concentration inequality~\cite{BL76, BL75} (see e.g.,~\cite[Section 9]{F05},~\cite[Section 6]{DD05}) that the limit exists in dimensions $d \geq 3$.

These qualitative properties are modified upon the addition of a quenched random external field to the system. This phenomenon can be observed by considering the specific case of the random-field Gaussian free field, that is the model~\eqref{eq:defmuLeta} with the interaction potential $V(x) = x^2/2$. In this setting, for any fixed realization of the external field $\eta$, the disordered random surface has a multivariate Gaussian distribution whose mean vector and covariance matrix are explicit. One can then rely on explicit computations to prove that, for almost every realization of the disorder, the law of the discrete gradient of the random surface converges weakly to an infinite-volume gradient Gibbs measure in dimensions $d \geq 3$ and does not converge in dimensions $d  \leq 2$. Similarly, one can show that, for almost every realization of the disorder, the law of the random surface converges weakly to an infinite-volume Gibbs measure in dimensions $d \geq 5$ and does not converge in dimensions $d \leq 4$ (see e.g.,~\cite[Appendix A.1]{CK12} where the non-existence of Gaussian Gibbs measures with finite first moment is established in dimensions $d = 3,4$). 

It is natural to expect that these properties persist for the random-field random surface model~\eqref{eq:defmuLeta} with a uniformly-convex interaction potential $V$ (for which the distribution of the random surface is not explicitly known). Indeed, Cotar and K\"{u}lske~\cite{CK12, CK15} proved the existence and uniqueness of infinite-volume translation-covariant gradient Gibbs measures in $d \geq 3$ with a fixed tilt (the uniqueness is established under the necessary condition that the corresponding annealed measure is ergodic with respect to spatial shifts). However, their work left open the question of the convergence of the finite-volume distributions to the infinite-volume Gibbs measures; see the discussion in~\cite[Appendix A]{CK12}. In this article, we work under the assumption that the law of the random field is symmetric and establish that, for almost every realization of the disorder, the distribution of the gradient of the random surface converges weakly to an infinite-volume gradient Gibbs measure in dimensions $d \geq 4$, and that the distribution of the random surface converges to an infinite-volume Gibbs measure in dimensions $d \geq 5$. This resolves the convergence question in the case of a random field with symmetric distribution and with uniformly elliptic interaction potential, apart from the gradient convergence in dimension $d=3$, which is expected but does not follow from our proof; see the discussion below Theorem~\ref{infinitevolgradgibbs}.
Further discussion of the relation of these results to the previous literature can be found in Section~\ref{sec.relresult}.

In order to state the results, we introduce a few additional notations and definitions. Write $\Lambda_L := \left\{ -L,  \ldots, L \right\}^d$, let $\vec{E}(\Zd)$ denote the set of oriented edges of the lattice $\Zd$ (see Section~\ref{section2.1}). For each function $\phi : \Zd \to \R$, we define its discrete gradient $\nabla \phi : \vec{E}(\Zd) \to \R$ by, for any $e = (x , y ) \in \vec{E}(\Zd)$, $\nabla \phi(e) := \phi(y)- \phi(x)$. We denote by $\Omega := \R^{\Zd}$ the set of functions defined on $\Zd$ and valued in $\R$, by $\mathcal{P} \left( \Omega \right)$ the set of probability measures on the space $\Omega $ and by $C_b\left( \Omega \right)$ the set of functions $f : \Omega \to \R$ which depend on finitely many coordinates, are continuous and bounded. For a measure $\mu \in \mathcal{P} \left( \Omega\right)$, we denote by $\left\langle \cdot \right\rangle_\mu$ the expectation with respect to $\mu$. Given a pair of measures $\mu_0 , \mu_1 \in \mathcal{P}\left( \Omega \right)$, we denote by $\Gamma (\mu_0 , \mu_1)$ the set of couplings between $\mu_0$ and $\mu_1$, that is, the set of probability measures on the product space $\Omega \times \Omega$ whose first and second marginals are equal to $\mu_0$ and $\mu_1$ respectively.

We denote by $\X$ the set of \emph{gradient fields}, that is 
\begin{equation} \label{def.Xgradfield}
    \X := \left\{ \chi : \vec{E}(\Zd) \to \R \, : \, \exists \phi : \Zd \to \R, \, \chi = \nabla \phi  \right\}.
\end{equation}
We let $\mathcal{P} \left( \X \right)$ be the set of probability measures on the space $\X$ and $C_b\left(\X \right)$ be the set of functions $f : \X \to \R$ which depend on finitely many coordinates, are continuous and bounded. For a measure $\mu \in \mathcal{P}\left( \X\right)$, we denote by $\left\langle \cdot \right\rangle_\mu$ the expectation with respect to $\mu$. Given a pair of measures $\mu_0 , \mu_1 \in \mathcal{P}\left( \X \right)$, we denote by $\Gamma_\nabla (\mu_0 , \mu_1)$ the set of couplings between $\mu_0$ and $\mu_1$.

For $y \in \Zd$, we denote by $\tau_y : \Omega \to \Omega$ the shift operator, i.e., $\tau_y \phi (x) = \phi (x-y)$. We extend the notation to the set of gradient fields by writing $\tau_y \chi (e) = \chi (e-y)$ and to the set of external fields by writing $\tau_y \eta(x) = \eta (x-y)$. For any measure $\mu \in \mathcal{P}\left( \Omega \right)$ (resp. $\mathcal{P}\left( \X \right)$), we denote by $(\tau_y)_* \mu$ the pushforward of the measure $\mu$ by the shift $\tau_y$. A sequence of probability measures $(\mu_n)_{n\geq 0} \in \mathcal{P} \left( \Omega \right)$ (resp. $\mathcal{P} \left( \X \right)$) converges weakly in $\mathcal{P}(\Omega)$ (resp. $\mathcal{P}(\X)$) to a measure $\mu \in \mathcal{P} \left( \Omega \right)$ (resp. $\mathcal{P} \left( \X \right)$) if, for any function $f \in C_b\left( \Omega \right)$ (resp. $f \in C_b\left( \X \right)$),
    \begin{equation*}
        \left\langle f \right\rangle_{\mu_n} \underset{n \to \infty}{\longrightarrow} \left\langle f \right\rangle_{\mu}.
    \end{equation*}
We remark that, for any finite subset $\Lambda \subseteq \Zd$ and any realization of the random field, one can see the measure $\mu_{\Lambda}^\eta$ as an element of the set $\mathcal{P} \left(\Omega\right)$ by extending the field $\phi$ by the value $0$ outside the set $\Lambda$. Similarly, if one wishes to study the discrete gradient of the random surface, one can see the measure $\mu_{\Lambda}^\eta$ as an element of the set $\mathcal{P} \left(\X\right)$ by considering the discrete gradient of the random surface and extending it to be equal to $0$ outside the set $\Lambda$. We may specify a boundary condition on the law of the disordered random surface as follows. For any $\psi : \Zd  \to \R$, $\eta :\Zd \to \R$ and $\Lambda \subseteq \Zd$ finite, let $\mu_\Lambda^{\eta , \psi} \in \mathcal{P}(\Omega)$ be defined by
\begin{equation*}
    \mu_\Lambda^{\eta , \psi} (\di \phi) := \frac{1}{Z_\Lambda^\eta}\exp \left( - H_{\Lambda}^{\eta} (\phi) \right) \prod_{v \in \Lambda} d \phi(v) \prod_{v \in \Zd \setminus \Lambda} \delta_{\psi(v)}(d \phi(v)),
\end{equation*}
where the symbol $\delta_x$ refers to the Dirac measure at the point $x \in \R$.

The first theorem of this article establishes that, for almost every realization of the random field, the sequence of measures $(\mu_{\Lambda_L}^\eta)$ converges weakly in $\mathcal{P}\left( \X \right)$ as $L$ tends to infinity in dimensions $d \geq 4$.

\begin{theorem}[Infinite-volume translation-covariant gradient Gibbs states in $d \geq 4$] \label{infinitevolgradgibbs}
Suppose $d \geq 4$, let $\lambda >0$ and assume that $(\eta(x))_{x \in \Zd}$ are independent with $\E \left[ \eta(x)^2 \right] =1$ and a symmetric distribution. Then, for almost every realization of the disorder $\eta$, the sequence of measures $(\mu_{\Lambda_L}^\eta)_{L \geq 0}$ converges weakly in~$\mathcal{P} \left( \X \right)$ to a measure $\mu^\eta_{\nabla} \in \mathcal{P} \left( \X \right)$ which satisfies the following properties:
\begin{itemize}
    \item \textit{Quantitative convergence:} There exist an exponent $\alpha > 0$ depending on the dimension $d$ and the ellipticity constants $c_-, c_+$ and a constant $C$ depending on the parameters $d , c_+ , c_-$ and $\lambda$ such that, for any side length $L \geq 1$,
    \begin{equation} \label{quantWassercoupling}
        \E \left[ \inf_{\kappa^\eta \in \Gamma_\nabla \left( \mu^\eta_{\Lambda_L} , \mu^\eta_{\nabla} \right)}\int_{\X \times \X} \sup_{e \in E(\Lambda_{L/2})}  \left|\chi_0(e) - \chi_1(e) \right|^2 \kappa^\eta \left( d\chi_0 , d \chi_1 \right) \right] \leq C L^{- \alpha}.
   \end{equation}
    \item \textit{Translation covariance:} For any vertex $y \in \Zd$ and almost every realization of the random field, we have
    \begin{equation} \label{translation covaraince grad Gibbs}
        \mu^{\tau_y \eta}_\nabla  = \left(\tau_y\right)_* \mu^{ \eta}_\nabla.
    \end{equation}
   \item \textit{Dobrushin–Lanford–Ruelle (DLR) equation:} For any finite set $\Lambda \subseteq \Zd$, any function $f \in C_b(\X)$ and almost every realization of the disorder, one has the identity
    \begin{equation} \label{eq:115211044}
        \int_{\X} \mu^{\eta}_\nabla(d \chi) \int_\X f(\nabla \phi) \mu_{\Lambda}^{\eta, \psi(\chi)}(d \phi)= \int_\X f(\chi) \mu^\eta_\nabla(d \chi) ,
    \end{equation}
    where $\psi(\chi)$ is any function whose discrete gradient is equal to $\chi$.
\end{itemize}
\end{theorem}

As mentioned above, the result of Theorem~\ref{infinitevolgradgibbs} should hold in any dimension $d \geq 3$, and the three dimensional case is not covered by our argument (while it is for instance included in the results of~\cite{CK12, CK15}). The technical reason behind is that, as established in~\cite[Theorem 2]{DHP20++}, the fluctuations of the disordered random surface distributed according to the measure~$\mu_{\Lambda_L}^\eta$ are of order $\sqrt{L}$ in three dimensions, while they are of order $\sqrt{\ln L}$ in four dimensions and are bounded in dimensions $5$ and higher. In dimension $d = 3$, these fluctuations are too large to be compensated by the $\alpha$-H\"{o}lder regularity provided by the De Giorgi-Nash-Moser regularity theory (see Proposition~\ref{paraNash}), which only holds for a small exponent $\alpha > 0$.

\medskip

The second theorem of this article states that, for almost every realization of the random field, the sequence of measures $(\mu_{\Lambda_L}^\eta)$ converges weakly in $\mathcal{P}\left( \Omega \right)$ as $L$ tends to infinity, and establishes the existence of infinite-volume translation-covariant Gibbs states in dimensions $d \geq 5$. 

\begin{theorem}[Infinite-volume translation-covariant Gibbs states in $d \geq 5$] \label{transcovgibbsstates}
Suppose $d \geq 5$, let $\lambda >0$ and assume that $(\eta(x))_{x \in \Zd}$ are independent with $\E \left[ \eta(x)^2 \right] =1$ and a symmetric distribution. Then, for almost every realization of the disorder $\eta$, the sequence of measures $(\mu_{\Lambda_L}^\eta)_{L \geq 0}$ converges weakly in $\mathcal{P} \left( \Omega \right)$ to a measure $\mu^\eta \in \mathcal{P} \left( \Omega \right)$ which satisfies the following properties:
\begin{itemize}
    \item \textit{Quantitative convergence:} There exists an exponent $\alpha > 0$ depending on the dimension $d$ and the ellipticity constants $c_-, c_+$ and a constant $C$ depending on the parameters $d , c_+ , c_-$ and $\lambda$ such that, for any $L \geq 1$,
    \begin{equation} \label{eq:09471104}
        \sup_{x \in \Lambda_{L/4}} \E \left[ \inf_{\kappa^\eta \in \Gamma \left( \mu^\eta_{\Lambda_L} , \mu^\eta \right)} \int_{\Omega \times \Omega} \left|\phi_0(x) -  \phi_1(x) \right|^2 \kappa^\eta \left( d\phi_0 , d \phi_1 \right) \right] \leq C L^{-\alpha}.
    \end{equation}
    \item \textit{Translation covariance:} For any vertex $y \in \Zd$ and almost every realization of the random field, we have
    \begin{equation} \label{eq:09481104}
        \mu^{\tau_y \eta} = \left(\tau_y\right)_* \mu^{ \eta}.
    \end{equation}
    \item \textit{Dobrushin–Lanford–Ruelle (DLR) equation:} For any finite set $\Lambda \subseteq \Zd$, any function $f \in C_b(\Omega)$ and almost every realization of the random field, one has the identity
    \begin{equation} \label{eq:11521104}
        \int_{\Omega} \mu^{\eta}(d \psi) \int_\Omega f(\phi) \mu_{\Lambda}^{\eta, \psi}(d \phi)= \int_\Omega f( \phi) \mu^\eta (d \phi).
    \end{equation}
\end{itemize}
\end{theorem}

We conclude this section by mentioning that a follow-up question to Theorem~\ref{infinitevolgradgibbs} and Theorem~\ref{transcovgibbsstates} is the convergence to a scaling limit for the random-field random surface model. It was pointed out in~\cite[Section 7]{DHP20++} that the ground state of random-field Gaussian free field coincides with the lattice \emph{membrane model}, and that it is natural to expect that the scaling limit is the continuum membrane model.

\subsection{Related results} \label{sec.relresult} 

\subsubsection{Random surfaces without disorder} 

The study of random surfaces (without disorder) was initiated in the 1970s by Brascamp, Lieb and Lebowitz~\cite{BLL75} who obtained sharp localization and delocalization estimates for the random surface under the uniform convexity assumption~\eqref{eq:V ellipticity}. In 1998, Funaki and Spohn~\cite{FS} established that the large-scale limit of the Langevin dynamics associated with the model is a deterministic quantity governed by a parabolic equation taking the form of a mean curvature flow. The scaling limit of the random surface was identified by Naddaf and Spencer~\cite{NS}, and reworked 
probabilistically by Giacomin, Olla and Spohn~\cite{GOS} in order to establish that the fluctuations of the space-time dynamics associated with the model behave like an infinite-dimensional Ornstein-Uhlenbeck process over large scales. We mention that, in the earlier work~\cite{BY}, Brydges and Yau established, among other results, a statement similar to the one of~\cite{NS} in a perturbative setting by using a renormalization group argument. In~\cite{DGI00}, Deuschel, Giacomin and Ioffe obtained large deviation estimates and concentration inequalities for the model. Sharp decorrelation estimates for the surface and its discrete gradient were established by Delmotte and Deuschel~\cite{DD05}. More recently, the techniques of~\cite{NS} were generalized by Miller~\cite{Mil} to finite-volume measures with general tilt. They were also combined with techniques of quantitative stochastic homogenization by Armstrong and Wu~\cite{AW19} to obtain the $C^2$ regularity of the surface tension of the $\nabla \phi$-model. 

The case when the potential $V$ is not (uniformly) convex may exhibit a different phenomenology, and a first-order phase transition from uniqueness to non-uniqueness of infinite-volume (gradient) Gibbs measure (with zero tilt) may be observed. In~\cite{BK07}, Biskup and Koteck{\'y} studied a class of models with potentials of the form
\begin{equation} \label{eq:1904}
    e^{-V_e (\chi)} = p e^{- \kappa_e \chi^2} +  (1-p) e^{- \kappa_e' \chi^2} ~\mbox{with}~ \kappa_e, \kappa_e' > 0, \, p\in [0 , 1] 
\end{equation}
and established the coexistence of infinite-volume, shift-ergodic gradient Gibbs measures when the parameters $\kappa_e, \kappa_e', p$ are chosen suitably. We also refer to~\cite{van2002first, van2005provable} where a related phenomenon was observed for spin systems with continuous symmetry and sufficiently nonlinear interaction potential. In~\cite{BS11}, Biskup and Spohn proved that, for a general category of non-convex potentials including~\eqref{eq:1904}, the scaling limit of the model is a Gaussian free field. Cotar and Deuschel~\cite{CD12} then proved the uniqueness of ergodic Gibbs measures, obtained sharp estimates on the decay of covariance and identified the scaling limit for a class of random surfaces with non-convex potentials. In~\cite{CDM09}, Cotar, Deuschel and M\"{u}ller established the strict convexity of the surface tension of the $\nabla \phi$-model with a non-convex potential which is a perturbation of a uniformly-convex potential. The uniform convexity of the surface tension for a class of random surfaces with non-convex potentials was obtained by Adams, Kotecky and M\"{u}ller~\cite{AKM16} through a renormalization group argument. Further details on random surfaces without external field can be found in the surveys of Funaki~\cite{F05}, Sheffield~\cite{Sh} and Velenik~\cite{V06}.

\smallskip

\subsubsection{Disordered random surfaces} \label{section1.2.2}

In their seminal work~\cite{IM75}, Imry and Ma predicted that the first-order phase transitions in low-dimensional spin systems are rounded upon the addition of a quenched random external field. Their predictions were confirmed for a large class of spin systems by Aizenman and Wehr~\cite{AW1989, AW89}, and recently quantified in a series of works in the case of the two-dimensional random-field Ising model~\cite{C18, AP19, ding20192exponential, AHP20, DW2020}. While the phenomenon has mostly been investigated in the context of compact spin spaces, a related effect occurs for the $\nabla \phi$-model when a quenched random field is suitably added to the system. 

Motivated by disordered spin systems, Bovier and K\"{u}lske~\cite{BK94, BK96} studied the effect of the addition of a $(d+1)$-dimensional disorder to the SOS-model for integer-valued random surfaces. Specifically, the disordered Hamiltonian they considered takes the following form:
\begin{equation*}
    H^\eta (\phi) := \sum_{x \sim y} |\phi(x) - \phi(y)| + \lambda \sum_{x , k} \eta_k (x) \indc_{\{\phi(x) = k \}},
\end{equation*}
for integer-valued surfaces $\phi: \Lambda \to \Z$ and where the sum is considered over the vertices $x \in \Zd$ and the integers $k \in \Z$, and the random variables $(\eta_k (x))_{k , x}$ are i.i.d. with mean zero and unit variance. Based on the renormalization group argument for the three dimensional random-field Ising model~\cite{BK88}, they established in~\cite{BK94} the existence of infinite-volume Gibbs measures in dimensions $d \geq 3$ when the strength of the external field and the temperature are small enough, and, by adapting the strategy of Aizenman and Wehr~\cite{AW89, AW1989}, they proved that infinite-volume, translation-covariant Gibbs states cannot exist for the model in dimensions $d \leq 2$ for any temperature and any strength of the random field.

Another class of random surfaces for which the phenomenon is observed is the model~\eqref{eq:defmuLeta} which was first studied by van Enter and K\"{u}lske~\cite{VK08}. They established that, contrary to the pure model (with $\eta \equiv 0$), infinite-volume gradient Gibbs measures cannot exist in two dimensions when the potential $V$ is even, continuously differentiable and has super-linear growth. They additionally obtained a quantitative estimate on the decay of correlation for infinite-volume translation-covariant gradient Gibbs measures in three dimensions.

In the article~\cite{CK12}, Cotar and K\"{u}lske studied two models of disordered random surfaces. The first one is the model~\eqref{eq:defmuLeta}. In the second one, the disorder is added to the model through the interacting potential as follows: they consider a collection of independent and identically distributed functions $\left( V_e \right)_{e \in E(\Zd)}$ satisfying suitable growth condition and study the random surface whose distribution is formally given by
\begin{equation} \label{Model B}
    \mu := \exp \left( - \sum_{\substack{e = (x, y)}} V_e(\phi(x) - \phi(y)) \right)/Z.
\end{equation}
In this setting, they investigate the question of the existence of infinite-volume, translation-covariant gradient Gibbs states for the model. They prove that such measures exist when the external field satisfies $\E \left[ \eta(x)\right] = 0$ in dimensions $d \geq 3$ for~\eqref{eq:defmuLeta} and in dimensions $d \geq 1$ for~\eqref{Model B}, as well as the existence of an infinite-volume surface tension under the same assumptions. They additionally establish that the infinite-volume surface tension does not exist for~\eqref{eq:defmuLeta} in dimensions $d \leq 2$, and that none of these quantities exist if the condition $\E \left[ \eta(x)\right] = 0$ is not satisfied.

Their result is closely related to Theorem~\ref{infinitevolgradgibbs}, but the techniques implemented in the two articles rely on different strategies. The article~\cite{CK12} uses surface tension bounds to obtain tightness of a suitable sequence of spatially averaged finite-volume gradient Gibbs measures. The techniques developed in the proof of Theorem~\ref{infinitevolgradgibbs} are based on a coupling argument for the Langevin dynamics associated with the model (initially introduced by Funaki and Spohn~\cite{FS}), together regularity estimates and upper bounds on the fluctuations of the random surface.

In the article~\cite{CK15}, Cotar and K\"{u}lske proved that infinite-volume translation-covariant and ergodic
gradient Gibbs states are unique when the tilt of the interface is specified in any dimension $d\geq 1$ for the model~\eqref{Model B}, and in any dimension $d \geq 3$ when the distribution of the disorder is symmetric for the model~\eqref{eq:defmuLeta}. They additionally obtained sharp decorrelation estimates for general functionals of the discrete gradient of the field for both models. We mention that some of the tools needed to prove the result are used in the proof of Theorem~\ref{transcovgibbsstates} (see Lemma~\ref{prop.thermalcorrelation} and Lemma~\ref{prop5.5} below).

Regarding the height of the random surface, K\"{u}lske and Orlandi~\cite{KO06} considered the model~\eqref{eq:defmuLeta} in two dimensions when the domain $\Lambda$ is a box. They establish that the typical height of the random surface at the center of the box is at least logarithmic in the side length of the box when the potential $V$ grows super-linearly for any realization of the disorder. In~\cite{KO08}, K\"{u}lske and Orlandi established that, when the potential $V$ has super linear growth and satisfies $V'' \leq C$, the random surface is delocalized (quantitatively) in two dimensions even when the interface is subjected to a strong $\delta$-pinning. They additionally obtained a lower bound on the height of the interface in dimensions $d \geq 3$ (exhibiting localization this time) as well as a lower bound on the fraction of pinned sites.

In the article~\cite{DHP20++}, the author, Harel and Peled obtained lower and upper bounds for the typical height of the random surface~\eqref{eq:defmuLeta} under the convexity assumption~\eqref{eq:V ellipticity}, exhibiting delocalization in dimensions $d \leq 4$ and localization in dimensions $d \geq 5.$

\subsection{Strategy of the proofs} \label{sec.strat}

In this section, we present of the main arguments developed in this article.

\subsubsection{Translation-covariant gradient Gibbs states} \label{section1.3.3}

To highlight the main ideas of the proof of Theorem~\ref{infinitevolgradgibbs}, we only describe the strategy in the case of the ground state of the model defined to be the minimizer of the variational problem
\begin{equation*}
    \inf_{\substack{v : \Lambda_L^+ \to \R \\ v = 0 \,\mathrm{on}\, \partial \Lambda_L}} \sum_{\substack{x , y \in \Lambda_L^+ \\ x \sim y}} V \left( v(y) - v(x) \right) - \lambda \sum_{x \in \Lambda_L} \eta(x) v(x).
\end{equation*}
We denote by $v_{L,\eta}$ the unique minimizer; it can be equivalently characterized as the solution of the discrete nonlinear elliptic equation
\begin{equation} \label{def.groundstate1706}
    \left\{ \begin{aligned}
    - \sum_{ y\sim x} V' \left( v_{L , \eta}(y)  - v_{L , \eta}(x) \right) &= \lambda \eta(x) &~\mbox{for}~x \in \Lambda_L, \\
    v_{L,\eta}(x) &= 0 &~\mbox{for}~x \in \partial \Lambda_L.
    \end{aligned}
    \right.
\end{equation}
In this setting, the question of the weak convergence of the disordered Gibbs measure $\mu^\eta_{\Lambda_L}$ (Theorem~\ref{infinitevolgradgibbs}) can be reformulated as follows: for almost every realization of the disorder, the sequence of functions $(\nabla v_{L , \eta})_{L \geq 0}$ converges pointwise as $L$ tends to infinity. The argument relies on the observation that, for any fixed realization of the disorder $\eta$, the difference $w_L := v_{2L , \eta} - v_{L , \eta}$ solves \emph{the uniformly elliptic linear equation} (see~\eqref{eq:11330408})
\begin{equation*}
    - \nabla \cdot \a \nabla w_L = 0 ~\mbox{in}~ \Lambda_L,
\end{equation*}
where the
environment $\a : E(\Lambda_L) \to \R$ is explicit, depends on the functions $v_{\eta, L}$ and $V$, and satisfies the uniform ellipticity estimates $c_- \leq \a \leq c_+$ (see~\eqref{def.environmenta}). The De Giorgi-Nash-Moser regularity estimate in the discrete setting (the version below is an immediate consequence of~\cite[Proposition 5.3 and Proposition 6.2]{delmotte1997inegalite}) implies that there exists an exponent $\alpha > 0$ such that
\begin{equation}
    \sup_{\substack{x , y \in \Lambda_{L/2} \\x \neq y}} \frac{\left| w_{L , \eta}(x) - w_{L , \eta}(x) \right|}{|x - y|^\alpha} \leq \frac{C}{L^\alpha} \left( \frac{1}{\left| \Lambda_L \right|} \sum_{x \in \Lambda_L} w_{L , \eta}(x)^2 \right)^\frac12,
\end{equation}
and thus, by restricting the supremum to the pairs of neighboring vertices,
\begin{equation} \label{eq:15270104}
    \sup_{e \in E \left( \Lambda_{L/2} \right)} \left| \nabla w_{L , \eta}(e) \right| \leq \frac{C}{L^\alpha} \left( \frac{1}{\left| \Lambda_L \right|} \sum_{x \in \Lambda_L} w_{L , \eta}(x)^2 \right)^\frac12.
\end{equation}
Using the results on the fluctuations of the random surface in the presence of a random field established in~\cite[Theorem 2]{DHP20++} and stated in Proposition~\ref{prop3.1009100bis} below (the result stated below only covers the case of the $\nabla \phi$-model, but its extension to the ground state follows from~\cite[Theorem 2]{DHP20++}), we can estimate the term in the right-hand side of~\eqref{eq:15270104} and obtain
\begin{equation} \label{eq:174804504}
    \E \left[ \sup_{e \in E \left(\Lambda_{L/2}\right)}\left| \nabla v_{2L,\eta}(e) - \nabla v_{L,\eta}(e) \right| \right]  \leq \frac{C}{L^\alpha} \left(\frac{1}{\left| \Lambda_L \right|} \sum_{x \in \Lambda_L} \E \left[ \left| w_{L , \eta}(x)\right|^2 \right]\right)^\frac 12 \leq \frac{C}{L^\alpha}.
\end{equation}
The inequality~\eqref{eq:174804504} implies that the sequence $ \nabla v_{2^{l},\eta}(e)$ converges almost surely by summing over the scales. 

The extension of the result from the ground state to the disordered Gibbs measure $\mu^\eta_{\Lambda_L}$ (as stated in Theorem~\ref{infinitevolgradgibbs}) is achieved by extending the argument from the elliptic setting to the parabolic setting, using the Langevin dynamics associated with the model (see Section~\ref{sectionLangevin}) and the notion of Wasserstein metric (see Section~\ref{sec.Wasserstein}) to estimate the distance between two probability distributions.

\subsubsection{Translation-covariant Gibbs states} \label{section1.3.4}

As in Section~\ref{section1.3.3}, we only outline the argument in the case of the ground state of the model. In this setting, Theorem~\ref{transcovgibbsstates} can be restated as follows: for almost every realization of the disorder, the sequence of functions $(v_{L , \eta})_{L \geq 0}$ converges pointwise as the side length $L$ tends to infinity. To establish the claim, we use standard results on the gradient of the Neumann Green's function on $\Zd$ (see Appendix~\ref{app.B1}) which imply that, for any $x \in \Lambda_{L/2}$ and any $\ell \in \N$, there exists a map $\mathbf{f}_{x + \Lambda_\ell} : \vec{E} \left(x+\Lambda_\ell\right) \to \R$ which is bounded by a constant depending only on the dimension and satisfies
\begin{align*}
    v_{2L , \eta}(x) - v_{L , \eta}(x) & = \sum_{e \in \vec{E}\left( x + \Lambda_{\ell} \right)} \mathbf{f}_{x+ \Lambda_\ell}(e) \left( \nabla v_{2L , \eta}(e) - \nabla v_{L , \eta}(e) \right) \\ & \quad + \frac{1}{\left| \Lambda_\ell \right|}\sum_{y \in (x + \Lambda_\ell)} \left( v_{2L , \eta}(y) -  v_{L , \eta}(y) \right).
\end{align*}
By choosing the side length $\ell$ to be a suitable power of $L$, we may use~\eqref{eq:174804504} to estimate the first term in the right-hand side, and a decorrelation estimate (see Section~\ref{section4.1}) to establish that the spatial averages of the functions $v_{2L , \eta}$ and $v_{L , \eta}$ over the cube $x + \Lambda_\ell$ are small in dimensions $d \geq 5$. Combining the different arguments, we obtain that there exists an exponent $\alpha >0$ such that, for any vertex $x \in \Lambda_{L/4}$,
\begin{equation*}
    \E \left[ \left| v_{2L , \eta}(x) - v_{L , \eta}(x) \right| \right] \leq \frac{C}{L^\alpha}.
\end{equation*}
The rest of the argument is then similar to the one of Theorem~\ref{infinitevolgradgibbs} presented in Section~\ref{section1.3.3}.

\subsection{Organisation of the article}
The rest of the article is organized as follows. Section~\ref{section2} introduces some additional notations and preliminary results used in the proofs of Theorem~\ref{infinitevolgradgibbs} and Theorem~\ref{transcovgibbsstates}. Section~\ref{Section5.1} and Section~\ref{Section4} contain the proofs of Theorem~\ref{infinitevolgradgibbs} and Theorem~\ref{transcovgibbsstates} respectively. Appendix~\ref{app.B1} contains the proof of a technical (and standard) result pertaining to the gradient of the Neumann Green's function used in the proof of Theorem~\ref{infinitevolgradgibbs}. Appendix~\ref{app.B2} builds upon standard results from the theory of elliptic regularity and contains the proof of the version of the De Giorgi-Nash-Moser regularity estimate used in the article.

\subsection{Convention for constants} Throughout this article, the symbols $C$ and $c$ denote positive constants which may vary from line to line, with $C$ increasing and $c$ decreasing. The symbol $\alpha$ denotes a positive exponent which may vary from line to line. These constants and exponents may depend only on the dimension $d$, the ellipticity constants $c_+$ and $c_-$ and the intensity $\lambda$ of the random field. We specify the dependency of the constants and exponents by writing, for instance, $C := C(d , c_+)$ to mean that the constant $C$ depends on the parameters $d$ and $c_+$. We simply write $C$ to mean $C := C(d , c_+ , c_- , \lambda)$.

\subsubsection*{Acknowledgements}
The research of the author was supported in part by Israel Science Foundation grants 861/15  and  1971/19  and  by  the  European  Research  Council starting  grant 678520 (LocalOrder). The author would like to thank the anonymous referees who provided useful and detailed comments on an earlier version of the manuscript, as well as for suggesting a simplification in the proofs of the estimates~\eqref{quantWassercoupling} and~\eqref{eq:09471104}.

\section{Notation and preliminary results} \label{section2}

\subsection{Notation} \label{section2.1} In this section, we introduce some general notations pertaining to lattices, discrete functions and probability distributions.

\subsubsection{General notation}  \label{section2.1general}
We consider the hypercubic lattice $\Zd$ in dimension $d \geq 4$ and denote by $\vec{E} (\Zd)$ (resp. $E (\Zd)$) its set of directed (resp. undirected) edges. Given a directed edge $e \in \vec{E} \left( \Zd\right)$, we denote by $x_e$ and $y_e$ the first and second endpoint of $e$. We denote by $\left|  \cdot  \right|$ the Euclidean norm on $\Zd$ and set $\left|  \cdot \right|_+ := \max(\left|  \cdot  \right| , 1)$. We write $x\sim y$ to denote that $\{x,y\}\in E(\Zd)$. We identify a subset $\Lambda\subset \Zd$ with the induced subgraph of $\Zd$ on $\Lambda$ and, in particular, denote by $\vec{E} \left( \Lambda \right)$ and $E \left( \Lambda \right)$ the sets of directed and undirected edges of $\Lambda$. We denote by $\partial \Lambda$ the \emph{external vertex boundary} of $\Lambda$, that is,
\begin{equation*}
    \partial \Lambda := \left\{ x \in \Zd \setminus \Lambda \, \colon \, \exists y \in \Lambda, y \sim x  \right\}.
\end{equation*}
We define $\Lambda^+:=\Lambda\cup\partial\Lambda$ and let $\left| \Lambda \right|$ be the cardinality of $\Lambda$. We write $\Lambda_L := \{-L, \ldots, L\}^d\subset\Z^d$ for any integer $L \ge 0$. This is extended to all $L \in [0 , \infty)$ by setting $\Lambda_L:=\Lambda_{\lfloor L\rfloor}$, where $\lfloor L\rfloor$ denotes the floor of $L$. Given a point $x \in \Zd$, we denote by $x + \Lambda_L := \left\{ y \in \Zd \, : \, y - x \in \Lambda_L \right\}$. A box $\Lambda$ is then a subset of $\Zd$ of the form $\Lambda := x + \Lambda_L$, for $x \in \Zd$ and $L \in \N$.

Given a subset $\Lambda \subseteq \Zd$ and a function $\phi : \Lambda \to \R$, we define the discrete gradient
$\nabla \phi(e) := \phi(y) - \phi(x)$ for each directed edge $e = (x , y) \in \vec{E}(\Lambda)$.
We extend the definition to \emph{time-dependent} functions $\phi : [t_- , t_+] \times \Lambda \to \R$ by setting
$\nabla \phi(t , e) := \phi(t , y) - \phi(t , x)$ for directed edges $e = (x , y) \in \vec{E}(G)$ and times $t\in [t_- , t_+]$.
In expressions which do not depend on the orientation of the edge, such as $|\nabla\phi(e)|^2$ or $V(\nabla\phi(e))$, we allow the edge $e$ to be undirected. Given a vertex $x \in \Zd$, we use the notation
\begin{equation*}
    \sum_{e \ni x} := \sum_{ e \in \vec{E}(\Zd) , x_e = x}.
\end{equation*}

Given a time-dependent function $\phi: [s , t] \times \Lambda \to \R$, we denote by $\partial_t \phi$ the derivative of $\phi$ with respect to the time variable. Given a vertex $x\in \Zd$ and a function $f$ depending on the random field $\eta$, we denote by $\frac{\partial f}{\partial \eta(x)}$ the derivative of the function $f$ with respect to the variable $\eta(x)$.

\subsubsection{Environments and elliptic operators}
A time-dependent environment is a map $\a : [0 , \infty) \times E(\Lambda) \to [0 , \infty)$. We say that an environment is uniformly elliptic if it satisfies the inequalities $0 < c_- \leq \a \leq c_+ < \infty$ for any $(t , e) \in [0 , \infty] \times E \left(\Zd \right)$. 

Given an environment $\a$, we define the time-dependent elliptic operator $-\nabla \cdot \a \nabla$ by the formula: for a function $\phi : [t_- , t_+] \times \Lambda \to \R$, a vertex $x\in \Lambda$ and a time $t \in [t_- , t_+]$,
\begin{equation} \label{eq:11330408}
    \nabla \cdot \a \nabla \phi(t , x) = \sum_{y \sim x} \a( t , \{x , y \}) (\phi(t , y) - \phi(t , x)).
\end{equation}

\subsubsection{Coupling and Wasserstein metric} \label{sec.Wasserstein}
As in the introduction, we let $\Omega = \R^{\Zd}$ be the space of infinite-volume surfaces. Following the presentation of Funaki--Spohn~\cite{FS}, we introduce the norm
\begin{equation*}
    \left| \phi \right|_1 := \left( \sum_{x \in \Zd} |\phi (x)|^2 e^{- |x|} \right)^{\frac 12}.
\end{equation*}
We denote by $\Omega_1 := \left\{ \phi \in \Omega \, : \, \left| \phi \right|_1 < \infty \right\}$ and let $\mathcal{P}_2 \left( \Omega_1\right)$ be the set of probability measures on $\Omega_1$ with finite second moment. For any pair $\mu_0 , \mu_1 \in \mathcal{P}_2 \left(\Omega_1\right)$, we define the Wasserstein distance between $\mu_0$ and $\mu_1$ by the formula
\begin{equation*}
    W_{\Omega} (\mu_0 , \mu_1) := \inf_{\kappa \in \Gamma(\mu_0 , \mu_1)} \left( \int_{\Omega_1 \times \Omega_1} \left| \phi_0 - \phi_1 \right|_1^2 d\kappa(\phi_0 , \phi_1) \right)^{\frac 12}.
\end{equation*}  
where $\Gamma(\mu_0 , \mu_1)$ denotes the set of probability measures on the space $\Omega_1 \times \Omega_1$ whose first and second marginals are the measures $\mu_0$ and $\mu_1$ respectively.

By \cite[Theorem 6.18]{villani2008optimal}, the space $(\mathcal{P}_2\left(\Omega_1 \right) , W_{\Omega})$ is metric, complete and separable. Additionally, if a sequence of measures $( \mu_n )_{n \geq 0}$ converges to a measure $\mu$ in the metric space $(\mathcal{P}_2\left(\Omega_1\right) , W_{\Omega})$ then it converges weakly in the sense that, for any $f \in C_b (\Omega)$, $\left\langle f \right\rangle_{\mu_n} \underset{n \to \infty}{\longrightarrow} \left\langle f \right\rangle_{\mu}$ (see~\cite[Theorem 6.9]{villani2008optimal}).

We denote by $\X$ the set of gradient fields, that is $$\X := \left\{ \chi : \vec{E}(\Zd) \to \R \, : \, \exists \phi : \Zd \to \R \, : \, \chi = \nabla \phi \right\}.$$
As it was the case in Section~\ref{section2.1general}, in expressions which do not depend on the orientation of the edge, we allow the edge $e$ to be undirected. We introduce the norms
\begin{equation*}
    \left| \chi \right|_1 := \left( \sum_{e \in \vec{E}(\Zd)} |\chi (e)|^2 e^{- |x_e|} \right)^{\frac 12}
\end{equation*}
We denote by $\X_1 := \left\{ \chi  \in \X \, : \, \left| \chi \right|_1 < \infty \right\}$, and by $\mathcal{P}_2\left(\X_1\right)$ the set of probability measures on $\X_1$ with finite second moment. For any pair $\mu_0 , \mu_1 \in \mathcal{P}_2\left(\X_1\right)$, we define the Wasserstein distance between $\mu_0$ and $\mu_1$ according to the formula
\begin{equation} \label{eq:10081104}
    W_{\X} (\mu_0 , \mu_1) := \inf_{\kappa \in \Gamma_\nabla(\mu_0 , \mu_1)} \left( \int_{\X_1 \times \X_1} \left|  \chi_0 -  \chi_1 \right|_1^2 d\kappa(\chi_0 , \chi_1) \right)^{\frac 12}
\end{equation}
It is also clear from the definition of the Wasserstein metric $W_{\X}$ and the norm $\left| \cdot \right|_1$ that, for any $y \in \Zd$ and any pair $\mu_0 , \mu_1 \in \mathcal{P}_2\left(\X_1\right)$,
\begin{equation} \label{eq:16481704}
    W_{\X} \left( \left( \tau_y \right)_* \mu_0 , \left( \tau_y \right)_* \mu_1 \right) \leq C e^{C|y|} W_{\X} \left(  \mu_0 , \mu_1 \right),
\end{equation}
where the notation $\left( \tau_y \right)_* \mu_0$ denotes the pushforward of the measure $\mu_0$ by the shift operator $\tau_y$ as introduced in Section~\ref{section1.11}.

By \cite[Theorem 6.18]{villani2008optimal}, the space $(\mathcal{P}_2\left(\X_1\right) , W_{\X})$ is metric, complete and separable. Additionally, if a sequence of measures $( \mu_n )_{n \geq 0}$ converges to a measure $\mu$ in the metric space $(\mathcal{P}_2\left(\X_1\right) , W_{\X})$ then it converges weakly in the sense that, for any $f \in C_b (\X)$, $\left\langle f \right\rangle_{\mu_n} \underset{n \to \infty}{\longrightarrow} \left\langle f \right\rangle_{\mu}$.

\subsection{Preliminary results} \label{section2.55}

In this section, we collect some tools pertaining to random surfaces and the regularity theory for solutions of parabolic equations which are used in the proofs of Theorem~\ref{infinitevolgradgibbs} and Theorem~\ref{transcovgibbsstates}.

\subsubsection{Langevin dynamics} \label{sectionLangevin}

The Gibbs measure $\mu_{\Lambda_L}^{\eta}$ (defined in~\eqref{eq:defmuLeta}) is naturally associated with the following dynamics.

\begin{definition}[Langevin dynamics]
Consider a box $\Lambda \subseteq \Zd$, a realization of the random field $\eta:\Lambda\to\R$ and let $\left\{ B_t(x) \, : \, x \in \Lambda\right\}$ be a collection of independent standard Brownian motions. The Langevin dynamics $\left\{ \phi_t(x) \, : \, x \in \Lambda\right\}$ is then defined to be the solution of the system of stochastic differential equations
\begin{equation} \label{eq:09262707bis}
    \left\{ \begin{aligned}
    \di \phi_t (x) &= \sum_{e \ni x} V' \left( \nabla \phi_t(e) \right) \di t + \lambda \eta(x)\di t + \sqrt{2}\,\di B_t(x) &&(t ,x) \in  (0 , \infty) \times \Lambda, \\
    \phi_t(x) &= 0 &&(t , x) \in  (0 , \infty) \times \partial \Lambda_L, \\
    \phi_0(x) &= 0 &&x \in \Lambda.
    \end{aligned} \right.
\end{equation}
\end{definition}

The Langevin dynamics~\eqref{eq:09262707bis} is stationary, reversible and ergodic with respect to the Gibbs measure $\mu_{\Lambda}^\eta$.

\subsubsection{De Giorgi-Nash-Moser regularity theory} 
The next statement we collect is the De Giorgi-Nash-Moser regularity estimate (following the works of De Giorgi~\cite{DG1}, Nash~\cite{nash1958continuity} and Moser~\cite{moser1961harnack, moser1964harnack}, see also~\cite[Chapter 8]{GT} and~\cite[Chapter 3]{han2011elliptic}) which states that the solutions of the parabolic equation $\partial_t u - \nabla \cdot \a \nabla u = 0$ are H\"{o}lder continuous when the environment $\a$ is uniformly elliptic. In the discrete setting, the result can be stated as follows.

\begin{proposition}[De Giorgi-Nash-Moser regularity for solutions of parabolic equation] \label{paraNash}
Fix an integer $L \geq 2$ and an environment $\a : [0 , \infty) \times E \left( \Zd\right) \to [0 , \infty)$ satisfying $c_- \leq \a \leq c_+$. Then, there exist an exponent $\alpha := \alpha (d , c_+, c_-) > 0$ and a positive constant $C := C (d , c_+, c_-) < \infty$ such that, for any function $u : (0 , L^2 ) \times \Lambda_L \to \R$ solution of the parabolic equation
\begin{equation*}
    \partial_t u - \nabla \cdot \a \nabla u = 0 \hspace{2mm} \mbox{in} \hspace{2mm} (0 , L^2 ) \times \Lambda_L,
\end{equation*}
one has, for any time $t \in [L^2/2, L^2]$,
\begin{equation} \label{NashParaest}
    \sup_{e \in E\left( \Lambda_{L/2}\right)} \left| \nabla u (t , e) \right| \leq \frac{C}{L^\alpha} \left( \frac{1}{L^{d+2}}\int_{0}^{L^2} \sum_{x \in \Lambda_L}\left| u(t , x) \right|^2 \, dt  \right)^{\frac12}.
\end{equation}
\end{proposition}

Proposition~\ref{paraNash} is a consequence of two standard results for solutions of parabolic equations: the mean value inequality (see~\cite[Proposition 3.2]{DD05}) and the increase of oscillations (see~\cite[Proposition 3.2]{DD05}). We include a short proof of the result in Appendix~\ref{app.B2} for completeness.

\subsubsection{Fluctuations for random-field random surfaces and the Brascamp-Lieb inequality}

In this section, we record the main result of the article~\cite{DHP20++} which provides an upper bound on the size of the fluctuations of the random surface in the presence of a random-field in dimensions $d \geq 4$. The bound stated in~\eqref{eq:2.81740} is sharp, and a corresponding lower bound is included in the statement of~\cite[Theorem 2]{DHP20++} (as well as upper and lower bounds on the fluctuations of the surface in the low dimensions $d \leq 3$).

\begin{proposition}[Theorem 2 of~\cite{DHP20++}]  \label{prop3.1009100bis}
Fix $\lambda > 0$, assume that $(\eta(x))_{x \in \Zd}$ are independent with $\E \left[ \eta(x)^2 \right] =1$ and a symmetric distribution. Then there exists a constant $C>0$ such that, for any  integer $L \geq 2$ and any $x \in \Lambda _L$,
\begin{equation} \label{eq:2.81740}
    \E \left[ \left\langle \phi(x) \right\rangle_{\mu_{\Lambda_L}^{\eta}}^2 \right] \leq \left\{ \begin{aligned} C \ln L ~\mbox{if}~d=4,\\
     C ~\mbox{if}~d\geq 5. 
    \end{aligned}\right.
\end{equation}
\end{proposition}

\begin{remark}
The statement of~\cite[Theorem 2]{DHP20++} estimates the variance over the disorder of the map $\left\langle \phi(x) \right\rangle_{\mu_{\Lambda_L}^{\eta}}$. The assumption that the distribution of the random field is symmetric ensures that $\E [ \left\langle \phi(x) \right\rangle_{\mu_{\Lambda_L}^{\eta}} ] = 0$ for any $x \in \Lambda_L$. The inequality~\eqref{eq:2.81740} can thus be deduced from the upper bound on the variance stated in~\cite[Theorem 2]{DHP20++}.
\end{remark}

Proposition~\ref{prop3.1009100bis} estimates the $L^2$-norm of \emph{the thermal expectation} of the height of the random surface. It is also natural to consider the value of the $L^2$-norm of the height of the surface under the annealed measure $\E\mu_{\Lambda_L}^{\eta}$. This quantity can be decomposed according to the formula
\begin{equation}\label{eq:total variance formula}
   \E\left[\left\langle \phi(x)^2 \right\rangle_{\mu_{\Lambda_L}^{\eta}}\right]  = \E\bigg[\Big\langle \big(\phi(x)  - \left\langle \phi(x) \right\rangle_{\mu_{\Lambda_L}^{\eta}} \big)^2\Big\rangle_{\mu_{\Lambda_L}^{\eta}} \bigg] + \E \left[ \left\langle \phi(x) \right\rangle_{\mu_{\Lambda_L}^{\eta}}^2 \right].
\end{equation}
The second term in the right-hand side is estimated by~\eqref{eq:2.81740}. An upper bound on the first term is provided by the Brascamp-Lieb inequality~\cite{BL76, BL75} which, in our setting, reads as follows: for any dimension $d \geq 3$, any side length $L \geq 1$, any vertex $x \in \Lambda_L$ and any realization of the random field $\eta$,
\begin{equation} \label{eq:11201008}
    \left\langle \left(\phi(x) - \left\langle \phi (x) \right\rangle_{\mu_{\Lambda_L}^{\eta}} \right)^2 \right\rangle_{\mu_{\Lambda_L}^{\eta}} \leq C.
\end{equation}
A combination of Proposition~\ref{prop3.1009100bis} and the inequality~\eqref{eq:11201008} implies the following estimate.

\begin{lemma}\label{2.111825}
Fix $\lambda > 0$, assume that $(\eta(x))_{x \in \Zd}$ are independent with $\E \left[ \eta(x)^2 \right] =1$ and a symmetric distribution. Then there exists a constant $C>0$ such that, for any  integer $L \geq 2$ and any $x \in \Lambda _L$,
    \begin{equation*} 
    \E \left[ \left\langle \phi(x)^2 \right\rangle_{\mu_{\Lambda_L}^{\eta}} \right] \leq \left\{ \begin{aligned} C \ln L ~\mbox{if}~d=4,\\
     C ~\mbox{if}~d\geq 5. 
    \end{aligned}\right.
\end{equation*}
\end{lemma}

\section{Translation-covariant gradient Gibbs states in dimensions $d \geq 4$} \label{Section5.1}

This section is devoted to the proof of Theorem~\ref{infinitevolgradgibbs}. The argument is decomposed into two steps. In Section~\ref{section3.1}, we establish a general coupling statement which allows to estimate the Wasserstein distance between the two measures $\mu_{\Lambda_0}^\eta$, $\mu_{\Lambda_1}^\eta$ for two boxes $\Lambda_0, \Lambda_1$ of comparable size and large intersection (Proposition~\ref{p.generalGNMcoupling}). Section~\ref{section3.2} then contains the proof of Theorem~\ref{infinitevolgradgibbs} based on the result established in Proposition~\ref{p.generalGNMcoupling}.

\subsection{A general coupling statement} \label{section3.1}

In this section, we construct a coupling between the two measures $\mu_{\Lambda_0}^\eta$, $\mu_{\Lambda_1}^\eta$ so as to minimize the $L^\infty$-norm of the difference of the gradients of the surfaces in the bulk of the two boxes $\Lambda_0$ and $\Lambda_1$.

\begin{proposition} \label{p.generalGNMcoupling}
Fix a dimension $d \geq 4$, let $\lambda > 0$, and assume that $(\eta(x))_{x \in \Zd}$ are independent with $\E \left[ \eta(x)^2 \right] =1$ and a symmetric distribution. There exists an exponent $\alpha := \alpha (d , c_+ , c_-)  >0$ such that the following holds. Let $L \geq 2$ be an integer and let $\Lambda_0 , \Lambda_1 \subseteq \Zd$ be two boxes such that $\Lambda_L \subseteq \Lambda_0 \cap \Lambda_1$ and $\Lambda_0 \cup \Lambda_1 \subseteq \Lambda_{4L}$. Then, for almost every $\eta$, there exists a coupling $\nu^\eta$ between the measures $\mu_{\Lambda_0}^\eta$ and $\mu_{\Lambda_1}^\eta$ such that
\begin{equation} \label{eq:12270903}
    \E \left[ \int_{\Omega \times \Omega} \sup_{e \in E \left(\Lambda_{L/2} \right)} \left| \nabla \phi_0(e) -  \nabla \phi_1(e)  \right|^2 \nu^\eta (d\phi_0 , d\phi_1) \right] \leq 
    C L^{-\alpha}.
\end{equation}
\end{proposition}

\begin{proof}
Fix an integer $L \geq 2$ and two boxes $\Lambda_0, \Lambda_1$ satisfying the assumptions of the proposition. We denote by $L_0$ and $L_1$ their respective side lengths and note that $L_0 , L_1 \leq 4 L$. We then fix a realization of the random field $\eta : \Zd \to \R$, and let $\phi_0$ and $\phi_1$ be two independent random surfaces distributed according to the measures $\mu^{\eta}_{\Lambda_0}$ and $\mu^{\eta}_{\Lambda_{1}}$ respectively. We consider the two Langevin dynamics (driven by the same Brownian motions)
\begin{equation} \label{eq:1106040822}
    \left\{ \begin{aligned}
    \di \phi_{0,t}(x) &= \sum_{e \ni x} V' \left( \nabla \phi_{0,t} ( e) \right) \di t +  \lambda \eta(x) \di t + \sqrt{2} \di B_t(x) &&(t , x) \in  (0 , \infty) \times \Lambda_0, \\
    \phi_{0,0}(x)  &= \phi_0(x) &&x \in \Lambda_0,\\
    \phi_{0,t}( x) &= 0 &&(t , x) \in (0 , \infty) \times \partial \Lambda_0,
    \end{aligned} \right.
\end{equation}
and
\begin{equation} \label{eq:1107040822}
    \left\{ \begin{aligned}
    \di \phi_{1,t}(x) &= \sum_{e \ni x} V' \left( \nabla \phi_{1,t} (e) \right) \di t + \lambda \eta(x) \di t + \sqrt{2} \di B_t(x) &&(t , x) \in  (0 , \infty) \times \Lambda_1, \\
    \phi_{1, 0}(x)  &= \phi_{1}(x) &&x \in \Lambda_{1},\\
    \phi_{1,t} (x) &= 0 &&(t , x) \in (0 , \infty) \times \partial \Lambda_1.
    \end{aligned} \right.
\end{equation}
Using that the Langevin dynamics~\eqref{eq:1106040822} (resp.~\eqref{eq:1107040822}) is invariant under Gibbs measure $\mu^{\eta}_{\Lambda_0}$ (resp. $\mu^{\eta}_{\Lambda_1}$), we see that the interface $\phi_{0,t}$ (resp. $\phi_{1,t}$) is distributed according to the measure $\mu^{\eta}_{\Lambda_0}$ (resp. $\mu^{\eta}_{\Lambda_1}$) for any time $t \geq 0$. We denote by $\mu$ the law of the joint process $\left( \phi_{0 , t}, \phi_{1,t} \right)_{t \geq 0}$ and by $\left\langle \cdot \right\rangle_\mu$ the expectation with respect to $\mu$.

Considering the difference between the equations~\eqref{eq:1106040822} and~\eqref{eq:1107040822} and using that the driving Brownian motions are the same, we obtain that the map $w(t, x) := \phi_{0,t}(x) - \phi_{1,t}(x)$ solves the equation
\begin{equation} \label{eq:08421404}
    \di w(t , x) = \left( \sum_{e \ni x} V' \left( \nabla \phi_{0,t} ( e) \right) - \sum_{e \ni x} V' \left( \nabla \phi_{1,t} ( e) \right) \right) \di t \hspace{5mm}\mbox{in} \hspace{5mm} [0 , \infty) \times \left( \Lambda_0 \cap \Lambda_1 \right).
\end{equation}
Using the definition~\eqref{eq:11330408} for the elliptic operator $- \nabla \cdot \a \nabla$ and the assumption $\Lambda_L \subseteq \Lambda_0 \cap \Lambda_1$, we see that the identity~\eqref{eq:08421404} implies
that the function $w(t, x) := \phi_{0,t}(x) - \phi_{1,t}(x)$ solves the parabolic equation
\begin{equation*}
    \partial_t w - \nabla \cdot \a \nabla w = 0\hspace{5mm} \mbox{in} \hspace{5mm} [0 , \infty) \times \Lambda_0 \cap \Lambda_1,
\end{equation*}
where the coefficient $\a$ is given by the formula
\begin{equation} \label{def.environmenta}
     \a (t  , e ) := \int_0^1 V'' \left( s \nabla \phi_{0,t}(e) + (1 - s) \nabla {\phi_{1,t}}(e) \right) \, d s.
\end{equation}
Using the uniform convexity of the potential $V$ stated in~\eqref{eq:V ellipticity} and applying the De Giorgi-Nash-Moser regularity estimate stated in Proposition~\ref{paraNash}, we obtain that there exists an exponent $\alpha:= \alpha(d , c_+ , c_-) > 0$ such that
\begin{align} \label{13462302}
    \sup_{e \in E \left( \Lambda_{L/2} \right)} \left| \nabla \phi_{0, L^2}(e) - \nabla \phi_{1, L^2}(e) \right|^2 & \leq \frac{C}{L^{\alpha}} \frac{1}{L^{d +2}} \int_0^{L^2} \sum_{x \in \Lambda_L}   \left( \phi_{0, t}(x) - \phi_{1, t}(x) \right)^2 \, dt \\
    & \leq \frac{C}{L^{\alpha}} \frac{1}{L^{d +2}} \int_0^{L^2} \sum_{x \in \Lambda_L}  \phi_{0, t}(x)^2 +  \phi_{1, t}(x)^2 \, dt. \notag
\end{align}
We then integrate over the measure $\mu$ and use the stationarity of the Langevin dynamics~\eqref{eq:1106040822} and~\eqref{eq:1107040822}. We obtain
\begin{align}
     \left\langle \sup_{e \in E \left( \Lambda_{L/2} \right)} \left| \nabla \phi_{0, L^2}(e) - \nabla \phi_{1, L^2}(e) \right|^2 \right\rangle_\mu & \leq \frac{C}{L^{\alpha}} \frac{1}{L^{d +2}} \int_0^{L^2} \sum_{x \in \Lambda_L}  \left\langle \phi_{0, t}(x)^2 \right\rangle_\mu +  \left\langle  \phi_{1, t}(x)^2 \right\rangle_\mu \, dt \\
    & = \frac{C}{L^{\alpha}} \frac{1}{L^{d +2}} \int_0^{L^2} \sum_{x \in \Lambda_L}  \left\langle \phi(x)^2 \right\rangle_{\mu^\eta_{\Lambda_{0}}} +  \left\langle \phi(x)^2 \right\rangle_{\mu^\eta_{\Lambda_{1}}} \, dt \notag \\
    & = \frac{C}{L^{\alpha}} \frac{1}{L^{d}} \sum_{x \in \Lambda_L}  \left\langle  \phi(x)^2 \right\rangle_{\mu^\eta_{\Lambda_{0}}} +  \left\langle \phi(x)^2 \right\rangle_{\mu^\eta_{\Lambda_{1}}}. \notag
\end{align}
Averaging over the disorder and using Lemma~\ref{2.111825}, we obtain
\begin{equation*}
    \E \left[  \left\langle \sup_{e \in E \left( \Lambda_{L/2} \right)} \left| \nabla \phi_{0, L^2}(e) - \nabla \phi_{1, L^2}(e) \right|^2 \right\rangle_\mu \right] \leq
\left\{ \begin{aligned}
    C \left( \ln L_0 + \ln  L_1 \right)L^{-\alpha} ~\mbox{if}~ d =4, \\
    CL^{-\alpha} ~\mbox{if}~ d \geq 5.
    \end{aligned} \right.
\end{equation*}
Using the assumption $L_0 , L_1 \leq 4L$ and reducing the value of the exponent $\alpha$ to absorb the logarithm, we deduce that, in any dimension $d \geq 4$,
\begin{equation} \label{eq:14490704}
    \E \left[ \left\langle \sup_{e \in E \left( \Lambda_{L/2} \right)} \left| \nabla \phi_{0, L^2}(e) - \nabla \phi_{1, L^2}(e) \right|^2 \right\rangle_{\mu} \right] \leq C L^{-\alpha}.
\end{equation}
We then define $\nu^\eta$ to be the law of the pair $\left(  \phi_{0, L^2},  \phi_{1, L^2}\right)$ to complete the proof of Proposition~\ref{p.generalGNMcoupling}.

\end{proof}

\subsection{Proof of Theorem~\ref{infinitevolgradgibbs}} \label{section3.2}

This section is devoted to the proof of Theorem~\ref{infinitevolgradgibbs} building upon the result of Proposition~\ref{p.generalGNMcoupling}. We recall the definitions of the Wasserstein distance~\eqref{eq:10081104} stated in Section~\ref{section2}.

\begin{proof}[Proof of Theorem~\ref{infinitevolgradgibbs}]
We first observe that Lemma~\ref{2.111825} and Proposition~\ref{p.generalGNMcoupling} imply the following inequality: for each fixed integer $L \geq 1$ and each pair of boxes $\Lambda_0 , \Lambda_1 \subseteq \Zd$ satisfying $\Lambda_L \subseteq \Lambda_0 \cap \Lambda_1 $ and $\Lambda_0 \cup \Lambda_1 \subseteq \Lambda_{4L}$
\begin{equation} \label{eq:1227090333}
    \E \left[ W_{\X} \left( \mu_{\Lambda_0}^\eta ,  \mu_{\Lambda_1}^\eta  \right)^2 \right] \leq 
    C L^{-\alpha}.
\end{equation}
Let us then fix $n \in \N$. Applying~\eqref{eq:1227090333} with the integer $L = 2^n$, the boxes $\Lambda_0 = \Lambda_{2^n}$ and $\Lambda_1 = \Lambda_{2^{n+1}}$, we obtain
\begin{equation} \label{eq:144909033}
   \E \left[ W_{\X} \left( \mu_{\Lambda_{2^n}}^\eta ,  \mu_{\Lambda_{2^{n+1}}}^\eta  \right)^2 \right] 
   \leq C 2^{-\alpha n}.
\end{equation}
The inequality~\eqref{eq:144909033} and the comparison estimate~\eqref{quantWassercoupling} imply that the sequence of measures $( \mu_{\Lambda_{2^n}}^\eta )_{n \in \N}$ is almost surely Cauchy in the metric space $(\mathcal{P}_2\left(\X_1\right) , W_{\X})$. Since this metric space is complete, it implies that the sequence converges almost surely and we denote the limiting measure by $\mu^\eta_\nabla \in \mathcal{P}_2\left(\X_1\right)$. It remains to verify that the measure $\mu^\eta_\nabla$ satisfies the properties listed in the statement of Theorem~\ref{infinitevolgradgibbs}.

\smallskip

We first verify the quantitative estimate~\eqref{quantWassercoupling}. By the inequality~\eqref{eq:14490704} and the gluing lemma (following the terminology of~\cite[Chapter 1]{Vil}), there exists a coupling $\pi^\eta_\nabla$ between the measures $\left(\mu_{\Lambda_{2^kL}}^\eta \right)_{k \geq 0}$ such that, for any $k \geq 0$,
\begin{equation} \label{eq:19071902}
    \E \left[ \left\langle \sup_{e \in \Lambda_{2^{k-1}L}} \left| \chi_{2^k L}(e) - \chi_{2^{k+1} L}(e)\right|^2 \right\rangle_{\pi^\eta_\nabla} \right] \leq C (2^{k}L)^{-\alpha},
\end{equation}
where $\left\langle \cdot \right\rangle_{\pi^\eta_\nabla}$ denotes the expectation with respect to $\pi^\eta_\nabla$ and $\chi_{2^k L}(e)$ is the gradient at the edge $e$ for the measure $\mu_{\Lambda_{2^kL}}^\eta$. The inequality~\eqref{eq:19071902} implies that, for almost every $\eta$, the gradients $\chi_{2^k L}$ converge pointwise almost surely with respect to the measure $\pi^\eta_\nabla$ to a random variable $\chi_\infty$ distributed according to $\mu^\eta_\nabla$. The triangle inequality additionally shows
\begin{equation} \label{eq:20091702}
    \E \left[ \left\langle \sup_{e \in \Lambda_{L/2}} \left| \chi_{L}(e) - \chi_{\infty}(e)\right|^2 \right\rangle_{\pi^\eta_\nabla} \right]^{\frac 12}  \leq \sum_{k=0}^\infty \E \left[ \left\langle \sup_{e \in \Lambda_{2^{k-1}L}} \left| \chi_{2^k L}(e) - \chi_{2^{k+1} L}(e)\right|^2 \right\rangle_{\pi^\eta_\nabla} \right]^{\frac 12} \leq C L^{-\alpha},
\end{equation}
which implies~\eqref{quantWassercoupling}.

We next verify the translation covariance~\eqref{translation covaraince grad Gibbs} of the measure $\mu^\eta_\nabla$. To this end, let us note that, by the definitions of the measure $\mu_{\Lambda}^{\eta}$ and of the shift operator $\tau_y$, one has the identity: for any point $y \in \Zd$, any realization of the random field $\eta$ and any side length $L \in \N$,
\begin{equation} \label{eq:16510704}
    \left(\tau_y \right)_* \mu_{\Lambda_{L}}^{\eta} =  \mu_{y + \Lambda_{L}}^{\tau_y \eta}.
\end{equation}
We then select an integer $n$ large enough (depending on $y$) so that $\Lambda_{2^{n-1}} \subseteq y + \Lambda_{2^n}$. Applying Proposition~\ref{p.generalGNMcoupling} with the integer $L := 2^{n-1}$ and the boxes $\Lambda_0 := \Lambda_{2^n}$ and $\Lambda_1 := y + \Lambda_{2^n}$, we obtain
\begin{equation} \label{eq:16580903}
    \E \left[ W_{\X} \left( \mu_{ \Lambda_{2^n}}^{\eta} , \mu_{y + \Lambda_{2^n}}^{\eta} \right)^2 \right] \leq C2^{-\alpha n}.
\end{equation}
Using the inequality~\eqref{eq:16580903} and the translation invariance of the law of the disorder $\eta$, we deduce that
\begin{align} \label{eq:16490704}
    \E \left[  W_{\X}\left( \mu^{\tau_y \eta}_\nabla , \mu_{y + \Lambda_{2^n}}^{\tau_y \eta} \right)^2 \right] & =  \E \left[  W_{\X}\left( \mu^{\eta}_\nabla , \mu_{y + \Lambda_{2^n}}^{\eta} \right)^2 \right] \\
    & \leq 2 \E \left[  W_{\X}\left( \mu^{\eta}_\nabla , \mu_{\Lambda_{2^n}}^{\eta} \right)^2 \right] + 2 \E \left[  W_{\X}\left( \mu_{\Lambda_{2^n}}^{\eta} , \mu_{y + \Lambda_{2^n}}^{ \eta} \right)^2 \right] \notag \\
    & \leq C 2^{-\alpha n} .\notag
\end{align}
Additionally, using the translation property~\eqref{eq:16481704} of the Wasserstein metric, we may write
\begin{equation} \label{eq:16500704}
    \E \left[  W_{\X}\left( \left(\tau_y\right)_* \mu^{\eta}_\nabla , \left(\tau_y\right)_* \mu_{\Lambda_{2^n}}^{\eta} \right)^2 \right] \leq C e^{C |y|} \E \left[  W_{\X}\left( \mu^{\eta}_\nabla , \mu_{\Lambda_{2^n}}^{\eta} \right)^2 \right] \leq C e^{C|y|}2^{-\alpha n}.
\end{equation}
Combining the identity~\eqref{eq:16510704} with the inequalities~\eqref{eq:16490704},~\eqref{eq:16500704}, we obtain, for any integer $n \in \N$ (chosen large enough depending on $y$),
\begin{align*} \label{eq:18010903}
    \E \left[ W_{\X}\left( \left(\tau_y\right)_* \mu^{\eta}_\nabla , \mu^{\tau_ y \eta}_\nabla \right)^2 \right] & \leq \E \left[ W_{\X}\left( \left(\tau_y \right)_* \mu^{\eta}_\nabla , \left(\tau_y\right)_* \mu_{\Lambda_{2^n}}^{\eta} \right)^2 \right] + \E \left[W_{\X}\left( \mu_{y + \Lambda_{2^n}}^{\tau_y \eta}, \mu^{\tau_y \eta}_\nabla \right)^2 \right] \\
    & \leq C (1 + e^{C|y|}) 2^{- \alpha n}.
\end{align*}
Take the limit $n \to \infty$ yields
\begin{equation} \label{eq:18380903}
    \E \left[ W_{\X}\left( \left(\tau_y\right)_* \mu^{\eta}_\nabla , \mu^{\tau_ y \eta}_\nabla \right)^2 \right] = 0,
\end{equation}
which implies the translation-covariance property~\eqref{translation covaraince grad Gibbs}.

There remains to prove the Dobrushin–Lanford–Ruelle (DLR) equation~\eqref{eq:115211044}. We first note that, by the definition of the finite-volume measure $\mu_{\Lambda_{2^n}}^\eta$, one has, for any function $f \in C_b \left( \X \right)$ and any subset $\Lambda \subseteq \Lambda_{2^n}$,
\begin{equation} \label{12011104}
    \int_{\Omega} \mu^{\eta}_{\Lambda_{2^n}}(d \phi_1) \int_\Omega f(\nabla \phi) \mu_{\Lambda}^{\eta, \phi_1 }(d \phi)= \int_\Omega f( \nabla \phi) \mu^\eta_{\Lambda_{2^n}} (d \phi).
\end{equation}
One then observes that the mapping $\phi_1 \mapsto \int_\Omega f(\nabla \phi) \mu_{\Lambda}^{\eta, \phi_1 }(d \phi)$ depends only on finitely many values of the (discrete) gradient of the map $\phi_1$ (the values on which $f$ depends and the values of the gradient on the edges which are in a finite neighborhood of the boundary $\partial \Lambda$). It is additionally continuous and bounded. Consequently, it belongs to the space $C_b(\X)$. Taking the limit $n \to \infty$ in both sides of the identity~\eqref{12011104} and using the almost-sure weak convergence of the sequence $\mu_{\Lambda_{2^n}}^{\eta}$ in the space $\mathcal{P} \left( \X \right)$, we obtain
\begin{equation*}
    \int_{\Omega} \mu^{\eta}_\nabla(d \chi) \int_\Omega f(\nabla \phi) \mu_{\Lambda}^{\eta, \psi(\chi) }(d \phi)= \int_\Omega f( \chi) \mu^\eta_\nabla (d \chi),
\end{equation*}
where $\psi(\chi)$ is any function satisfying $\nabla \psi(\chi) = \chi$. The proof of~\eqref{eq:115211044}, and thus of Theorem~\ref{infinitevolgradgibbs}, is complete.
\end{proof}

\section{Translation-covariant Gibbs states in dimensions $d \geq 5$} \label{Section4}

This section is devoted to the proof of Theorem~\ref{transcovgibbsstates}. The argument is decomposed in three steps, which correspond to Section~\ref{section4.1}, Section~\ref{Section4.2} and Section~\ref{section4.3}. In Section~\ref{section4.1}, we prove a decorrelation estimate for the spatial average of the height of the random surface in dimensions~$d \geq 5$. In Section~\ref{Section4.2}, we combine the decorrelation inequality established in Section~\ref{section4.1} with the coupling estimate for the gradient of the surface proved in Proposition~\ref{p.generalGNMcoupling}, and establish an upper bound for the Wasserstein distance between the two measures $\mu_{\Lambda_0}^\eta$ and $\mu_{\Lambda_1}^\eta$ where $\Lambda_0$ and $\Lambda_1$ are two boxes of similar size and large intersection. Section~\ref{section4.3} contains the proof of Theorem~\ref{transcovgibbsstates}, building upon the results of Section~\ref{Section4.2}.

\subsection{A decorrelation estimate for the random surface} \label{section4.1}
We first establish an upper bound on the $L^2$-norm of the spatially averaged height of the random surface in a box $\Lambda_0$ under the annealed measure $\E \mu_{\Lambda}^\eta$. We note that the arguments used in the proof of Lemma~\ref{lemma15061003} are similar to the ones of~\cite[Theorem 3.1]{CK15}.

\begin{lemma} \label{lemma15061003}
    Fix a dimension $d \geq 5$, let $\lambda > 0$, and assume that $(\eta(x))_{x \in \Zd}$ are independent with $\E \left[ \eta(x)^2 \right] =1$ and a symmetric distribution. There exists a constant $C <\infty$ such that for any box $\Lambda_0$ of side length $\ell \geq 1$ and any box $\Lambda$ satisfying $\Lambda_0 \subseteq \Lambda$, 
    \begin{equation} \label{eq:08381204}
    \E \left[  \left\langle \left| \frac{1}{\left| \Lambda_0\right|} \sum_{x \in \Lambda_{0}} \phi(x) \right|^2 \right\rangle_{\mu^\eta_\Lambda} \right] \leq \frac{C}{\ell^{d-4}}.
\end{equation}
\end{lemma}

The proof of Lemma~\ref{lemma15061003} requires to use the following two results for which we refer to the articles~\cite{GO2, CK15}. The first one provides an estimate on the covariance of the variables $\phi(x)$ and $\phi(y)$ when the surface $\phi$  is distributed according to the measure $\mu_\Lambda^\eta$ and applies uniformly over the realizations of the disorder $\eta$. Before stating the result, we recall the notation $|\cdot|_+ := \max \left( |\cdot| , 1 \right). $

\begin{lemma}[Proposition 2.3 (ii) of~\cite{CK15}] \label{prop.thermalcorrelation}
Fix a dimension $d \geq 3$. There exists a constant $C:= C(d , c_+ , c_-) < \infty$ such that for any $\lambda > 0$, any realization of the disorder $\eta : \Lambda \to \R$, and any pair of points $x , y \in \Lambda$,
\begin{equation} \label{eq:15231003}
    \left| \left\langle \phi(x) \phi(y) \right\rangle_{\mu^\eta_{\Lambda}} - \left\langle \phi(x) \right\rangle_{\mu^\eta_{\Lambda}} \left\langle \phi(y) \right\rangle_{\mu^\eta_{\Lambda}} \right| \leq \frac{C}{|x - y|_+^{d-2}}.
\end{equation}
\end{lemma}

The next statement we collect is a covariance inequality for functions of the random field. Estimates of this nature were used in stochastic homogenization by Naddaf and Spencer~\cite{NS2}, and then by Gloria and Otto~\cite{GO1, GO2} to obtain optimal quantitative estimates on the first-order corrector. Lemma~\ref{prop5.5} below can, in particular, be deduced from~\cite[Lemma 3]{GO2}. In the setting of random-field random surfaces, it was used by Cotar and K\"{u}lske in~\cite[Proposition 2.4]{CK15}.

\begin{lemma} \label{prop5.5}
Assume that $(\eta(x))_{x \in \Zd}$ are independent with $\E [ \eta(x)^2 ] =1$ and fix a box $\Lambda \subseteq \Zd$. 
For any pair $f , g$ of Lipschitz and differentiable functions of the collection $\left( \eta_v \right)_{v \in \Lambda} \in \R^{\Lambda}$, one has the estimate
    \begin{equation*}
        \left| \E \left[f  g\right] - \E \left[f  \right] \E \left[g\right] \right| \leq \sum_{z \in \Lambda} \sup_{\eta \in \R^{\Lambda} } \left| \frac{\partial f}{\partial \eta(z)} \right| \times  \sup_{\eta \in \R^{\Lambda} } \left| \frac{\partial g}{\partial \eta(z)} \right|.
    \end{equation*}
\end{lemma}

Using the two previous lemmas, we are now ready to show Lemma~\ref{lemma15061003}.

\begin{proof}[Proof of Lemma~\ref{lemma15061003}]
We first prove the following inequality: for any box $\Lambda \subseteq \Zd$ and any pair of vertices $x , y \in \Lambda$,
\begin{equation} \label{eq:ineqlemma5.3}
    \left| \E \left[ \left\langle \phi(x) \phi(y) \right\rangle_{\mu_\Lambda^\eta} \right] \right| \leq \frac{C}{|x - y|_+^{d-4}}.
\end{equation}
Using that the inequality~\eqref{eq:15231003} of Lemma~\ref{prop.thermalcorrelation} is valid for any realization of the random field, the estimate~\eqref{eq:ineqlemma5.3} is equivalent to the following upper bound: for any box $\Lambda \subseteq \Zd$ and any pair of points $x , y \in \Lambda$,
\begin{equation} \label{eq:16311003}
    \left| \E \left[ \left\langle \phi(x) \right\rangle_{\mu_\Lambda^\eta} \left\langle \phi(y) \right\rangle_{\mu_\Lambda^\eta} \right] \right| \leq \frac{C}{|x - y|_+^{d-4}}.
\end{equation}
To prove the inequality~\eqref{eq:16311003}, we note that the functions $\eta \mapsto \left\langle \phi(x) \right\rangle_{\mu_\Lambda^\eta}$ and $\eta \mapsto \left\langle \phi(y) \right\rangle_{\mu_\Lambda^\eta}$ depend only on the values of the field $\eta$ in the box $\Lambda$ and are differentiable. An explicit computation yields the formulae, for any point $z \in \Lambda$,
\begin{equation*}
    \frac{\partial \left\langle \phi(x) \right\rangle_{\mu_\Lambda^\eta}}{\partial \eta(z)} = \lambda \left\langle \phi(x) \phi(z) \right\rangle_{\mu_\Lambda^\eta} - \lambda \left\langle \phi(x) \right\rangle_{\mu_\Lambda^\eta} \left\langle  \phi(z) \right\rangle_{\mu_\Lambda^\eta}
\end{equation*}
and
\begin{equation*}
    \frac{\partial \left\langle \phi(y) \right\rangle_{\mu_\Lambda^\eta}}{\partial \eta(z)} = \lambda\left\langle \phi(y) \phi(z) \right\rangle_{\mu_\Lambda^\eta} -  \lambda\left\langle \phi(y) \right\rangle_{\mu_\Lambda^\eta} \left\langle  \phi(z) \right\rangle_{\mu_\Lambda^\eta}.
\end{equation*}
Applying Proposition~\ref{prop.thermalcorrelation}, we obtain that the functions $\eta \mapsto \left\langle \phi(x) \right\rangle_{\mu_\Lambda^\eta}$ and $\eta \mapsto \left\langle \phi(y) \right\rangle_{\mu_\Lambda^\eta}$ are Lipschitz and satisfy, for any vertex $z \in \Lambda$ and any realization of the random field $\eta : \Lambda \to \R$,
\begin{equation} \label{eq:17241003}
    \left| \frac{\partial \left\langle \phi(x) \right\rangle_{\mu_\Lambda^\eta}}{\partial \eta(z)} \right| \leq \frac{C}{|x - z|_+^{d-2}}  \hspace{3mm} \mbox{and} \hspace{3mm} \left| \frac{\partial \left\langle \phi(y) \right\rangle_{\mu_\Lambda^\eta}}{\partial \eta(z)} \right| \leq \frac{C}{|y - z|_+^{d-2}}.
\end{equation}
We can thus apply Lemma~\ref{prop.thermalcorrelation} and Lemma~\ref{prop5.5} with the functions $f : \eta \mapsto \left\langle \phi(x) \right\rangle_{\mu_\Lambda^\eta}$ and $g : \eta \mapsto \left\langle \phi(y) \right\rangle_{\mu_\Lambda^\eta}$, and note that the symmetry assumption on the random field implies the identities $\E [ \left\langle \phi(x) \right\rangle_{\mu_\Lambda^\eta} ] = \E [ \left\langle \phi(y) \right\rangle_{\mu_\Lambda^\eta} ] = 0$. We obtain
\begin{align*}
    \left| \E \left[ \left\langle \phi(x) \right\rangle_{\mu_\Lambda^\eta} \left\langle \phi(y) \right\rangle_{\mu_\Lambda^\eta} \right] \right| & \leq C \sum_{z \in \Lambda} \frac{1}{|x-z|_+^{d-2}} \frac{1}{|y-z|_+^{d-2}} \\
    & \leq  C \sum_{z \in \Zd} \frac{1}{|x-z|_+^{d-2}} \frac{1}{|y-z|_+^{d-2}} \\
    & \leq  \frac{C}{|x-y|_+^{d-4}}.
\end{align*}
The proof of the inequality~\eqref{eq:16311003}, and thus of~\eqref{eq:ineqlemma5.3}, is complete. We next prove the inequality~\eqref{eq:08381204}.
To this end, we use the inequality~\eqref{eq:ineqlemma5.3} and write
\begin{align} \label{eq:15002903}
   \E \left[\left\langle \left| \frac{1}{\left| \Lambda_0 \right|}\sum_{y \in  \Lambda_0} \phi(y) \right|^2 \right\rangle_{\mu_{\Lambda}^\eta} \right] & = \E \left[\left\langle \left| \frac{1}{\left| \Lambda_0 \right|}\sum_{y \in  \Lambda_0}  \phi(y) \right|^2 \right\rangle_{\mu_{\Lambda}^\eta} \right]
   \\ 
   & = \frac{1}{\left| \Lambda_0 \right|^2} \sum_{y , z \in \Lambda_0} \E \left[\left\langle  \phi(y) \phi(z) \right\rangle_{\mu_{\Lambda}^\eta} \right] \notag \\
   & \leq \frac{C}{\ell^{2d}} \sum_{y , z \in \Lambda_0} \frac{1}{|y - z|_+^{d-4}} \notag \\
   & \leq \frac{C}{\ell^{d-4}}. \notag
\end{align}
The proof of Lemma~\ref{lemma15061003} is complete.
\end{proof}

\subsection{A general coupling statement} \label{Section4.2}

In this section, we establish a general coupling statement to estimate the Wasserstein distance between the two measures $\mu_{\Lambda_0}^\eta$, $\mu_{\Lambda_1}^\eta$where $\Lambda_0, \Lambda_1$ are two boxes of comparable size and large intersection.

\begin{proposition} \label{p.generalGNMcoupling22}
Fix a dimension $d \geq 5$, let $\lambda > 0$, and assume that $(\eta(x))_{x \in \Zd}$ are independent with $\E [ \eta(x)^2 ] =1$ and a symmetric distribution. There exists an exponent $\alpha:= \alpha(d , c_+ , c_-)  >0$ such that the following statement holds. Let $L \geq 2$ be an integer and let $\Lambda_0 , \Lambda_1 \subseteq \Zd$ be two boxes such that $\Lambda_L \subseteq \Lambda_0 \cap \Lambda_1$ and $\Lambda_0 \cup \Lambda_1 \subseteq \Lambda_{4L}$. Then, one has the estimate
\begin{equation*}
    \E \left[ W_{\Omega} \left( \mu_{\Lambda_0}^\eta , \mu_{\Lambda_1}^\eta  \right)^2 \right] \leq 
    C L^{-\alpha}.
\end{equation*}
\end{proposition}

\begin{proof}
We fix an integer $L \geq 2$, two boxes $\Lambda_0 , \Lambda_1$ satisfying the assumptions of the proposition,  and let $\nu^\eta$ be the coupling between the measures $\mu_{\Lambda_0}^\eta$ and $\mu_{\Lambda_1}^\eta$ provided by Proposition~\ref{p.generalGNMcoupling}.
By Proposition~\ref{propB1}, we know that, for any $x \in \Zd$ and any $\ell \in \N$, there exists a function $\mathbf{f}_{x + \Lambda_\ell} : \vec{E} \left( x + \Lambda_\ell \right) \to \R$ and a constant $C := C(d) < \infty$ such that, for any $ \phi : (x + \Lambda_\ell) \to \R$,
 \begin{equation} \label{eq:1431080444}
       \phi(x) =  \sum_{e \in \vec{E} \left( x + \Lambda_\ell \right)} \mathbf{f}_{x + \Lambda_\ell}(e) \nabla \phi (e)  + \frac{1}{\left| \Lambda_\ell \right|} \sum_{y \in (x + \Lambda_\ell)} \phi(y),
    \end{equation}
and, for any edge $e \in \vec{E} \left( x + \Lambda_\ell\right),$
\begin{equation} \label{eq:14320804}
    \left| \mathbf{f}_{x + \Lambda_\ell}(e) \right| \leq C.
\end{equation}
Let $\ell$ be integer such that $\ell \leq L/4$ and fix a vertex $x \in \Lambda_{L/4}$. Noting that $x + \Lambda_{\ell} \subseteq \Lambda_{L/2}$, we can apply the identity~\eqref{eq:1431080444} and the inequality~\eqref{eq:14320804} with the box $\Lambda := (x + \Lambda_{\ell})$. We obtain, for any pair of functions $\phi_0, \phi_1 : \Lambda_L \to \R$,
\begin{align} \label{eq:14012903}
     \left| \phi_0(x) - \phi_1(x) \right|&  \leq C \left| \Lambda_{\ell} \right|  \sup_{e \in E \left( x + \Lambda_{\ell}\right) } \left| \nabla \phi_0(e) - \nabla \phi_1 (e) \right| \\ & \quad + \left| \frac{1}{\left| \Lambda_{\ell} \right|}\sum_{y \in  (x +\Lambda_{\ell})} \phi_0(y) \right| + \left| \frac{1}{\left| \Lambda_{\ell} \right|}\sum_{y \in  (x +\Lambda_{\ell})} \phi_1(y) \right|. \notag
\end{align}
Taking the expectation with respect to measure $\nu^\eta$ (using that the marginals of this distribution are $\mu_{\Lambda_0}^\eta$ and $\mu_{\Lambda_1}^\eta$) and then over the disorder $\eta$, we obtain
\begin{align*}
    \lefteqn{ \sup_{x \in \Lambda_{L/4}}\E \left[ \int_{\Omega \times \Omega} \left| \phi_0(x) - \phi_1(x)  \right|^2 \nu^\eta (d\phi_0 , d\phi_1) \right]} \qquad & \\ & \leq C \ell^d \E \left[ \int_{\Omega \times \Omega} \sup_{e \in E \left( \Lambda_{L/2}\right) } \left| \nabla \phi_0(e) - \nabla \phi_1(e)  \right|^2 \nu^\eta (d\phi_0 , d\phi_1) \right] \\
    & \quad +3 \sup_{x \in \Lambda_{L/4}} \E \left[ \left\langle \left| \frac{1}{\left| \Lambda_{\ell} \right|}\sum_{y \in  (x +\Lambda_{\ell})} \phi(y) \right|^2 \right\rangle_{\mu_{\Lambda_0}^\eta} \right] + 3\sup_{x \in \Lambda_{L/4}} \E \left[ \left\langle \left| \frac{1}{\left| \Lambda_{\ell} \right|}\sum_{y \in  (x +\Lambda_{\ell})} \phi(y) \right|^2 \right\rangle_{\mu_{\Lambda_{1}^{\eta}}} \right].
\end{align*}
The first term in the right-hand side can be estimated thanks to Proposition~\ref{p.generalGNMcoupling}, the second and third terms can be bounded from above by using Lemma~\ref{lemma15061003}. We obtain
\begin{equation*}
    \sup_{x \in \Lambda_{L/4}} \E \left[ \int_{\Omega \times \Omega} \left| \phi_0(x) -  \phi_1(x)  \right|^2 \nu^\eta (d\phi_0 , d\phi_1) \right] \leq C \ell^d L^{-\alpha} + C\ell^{4-d}.
\end{equation*}
Optimising over the value of $\ell$ (for instance choosing $\ell = L^{\alpha/(2d)}$) and reducing the value of the exponent $\alpha$, we deduce that
\begin{equation} \label{eq:16491204}
    \sup_{x \in \Lambda_{L/4}} \E \left[ \int_{\Omega \times \Omega} \left| \phi_0(x) -  \phi_1(x)  \right|^2 \nu^\eta (d\phi_0 , d\phi_1) \right] \leq C L^{-\alpha}.
\end{equation}
A combination of Lemma~\ref{2.111825} and~\eqref{eq:16491204} shows
\begin{equation*}
    \E \left[ W_{\Omega} \left( \mu_{\Lambda_0}^\eta ,  \mu_{\Lambda_{1}}^\eta  \right)^2 \right] \leq C L^{-\alpha}.
\end{equation*}
\end{proof}

\subsection{Proof of Theorem~\ref{transcovgibbsstates}} \label{section4.3}

This section contains the proof of Theorem~\ref{transcovgibbsstates} based on the coupling statement established in Section~\ref{Section4.2}.

\begin{proof}[Proof of Theorem~\ref{transcovgibbsstates}]
Fix an integer $n \in \N$. Applying Proposition~\ref{p.generalGNMcoupling22} with $L = 2^n$, the boxes $\Lambda_0 = \Lambda_{2^n}$ and $\Lambda_1 = \Lambda_{2^{n+1}}$, we obtain
\begin{equation} \label{eq:14490903}
   \E \left[ W_{\Omega} \left( \mu_{\Lambda_{2^n}}^\eta ,  \mu_{\Lambda_{2^{n+1}}}^\eta  \right)^2 \right] 
   \leq C 2^{-\alpha n}.
\end{equation}
The inequality~\eqref{eq:14490903} implies that the sequence of measures $( \mu_{\Lambda_{2^n}}^\eta )_{n \in \N}$ is almost surely Cauchy in the metric space $(\mathcal{P}_2\left(\Omega_1\right) , W_{\Omega})$. Since this metric space is complete, it implies that it converges almost surely to a probability distribution $\mu^\eta \in \mathcal{P}_2\left(\Omega_1\right)$. It remains to verify that the measure $\mu^\eta$ satisfies the properties listed in the statement of Theorem~\ref{transcovgibbsstates}.
First, the inequality~\eqref{eq:14490903} implies the estimate: for any $L \geq 2$,
\begin{equation} \label{eq:16091204}
     \E \left[  W_{\Omega} \left( \mu_{\Lambda_L}^\eta ,  \mu^\eta  \right)^2 \right] \leq CL^{-\alpha}.
\end{equation}
We then fix an integer $L \geq 2$ and let $n$ be the integer satisfying $2^{n - 1} < L \leq 2^n$. Applying Proposition~\ref{p.generalGNMcoupling22} with the boxes $\Lambda_0 := \Lambda_L$ and $\Lambda_1 := \Lambda_{2^n}$, we obtain
\begin{equation} \label{eq:15350903}
    \E \left[  W_{\Omega} \left( \mu_{\Lambda_{2^n}}^\eta ,  \mu_{\Lambda_{L}}^\eta  \right)^2 \right] \leq CL^{-\alpha}.
\end{equation}
Consequently, by the triangle inequality for the Wasserstein metric,
\begin{equation*}
    \E \left[  W_{\Omega} \left( \mu_{\Lambda_{L}}^\eta ,  \mu^\eta  \right)^2 \right] \leq 2\E \left[  W_{\Omega} \left( \mu_{\Lambda_{2^n}}^\eta ,  \mu_{\Lambda_{L}}^\eta  \right)^2 \right] + 2\E \left[  W_{\Omega} \left( \mu_{\Lambda_{2^n}}^\eta ,  \mu^\eta  \right)^2 \right] \\
     \leq C L^{-\alpha}.
\end{equation*}
The proof of~\eqref{eq:16091204} is complete.

The proofs of the translation covariance~\eqref{eq:09481104} and the Dobrushin–Lanford–Ruelle (DLR) equation~\eqref{eq:11521104} for the Gibbs states $\mu^\eta$ are a notational modification of the corresponding proofs for the gradient Gibbs states $\mu_{\nabla}^\eta$ established in Section~\ref{section3.2}. We thus omit the details.

We finally prove the quantitative coupling estimate~\eqref{eq:09471104}. Using~\eqref{eq:16491204} and the gluing lemma (see~\cite[Chapter 1]{Vil}), we obtain that there exists a coupling $\pi^\eta$ between the measures $\left( \mu_{\Lambda_{2^k L}}^\eta \right)_{k \geq 0}$ such that
\begin{equation} \label{17151204}
    \sup_{x \in 2^{k-1} L} \E \left[ \left\langle \left|\phi_{2^kL}(x) -  \phi_{2^{k+1}L}(x) \right|^2 \right\rangle_{\pi^\eta} \right] \leq C (2^k L)^{-\alpha}, 
\end{equation}
where $\left\langle \cdot \right\rangle_{\pi^\eta}$ denotes the expectation with respect to $\pi^\eta$ and $\phi_{2^k L}(x)$ is the height of the surface at the vertex~$x$ for the measure $\mu_{\Lambda_{2^kL}}^\eta$. The inequality~\eqref{eq:19071902} implies that, for almost every $\eta$, the random variables $\phi_{2^k L}$ converge pointwise almost surely with respect to the measure $\pi^\eta$ to a random variable $\phi_\infty$ distributed according to $\mu^\eta$. Summing the inequality~\eqref{17151204} over the integers $k \in \N$ and using the triangle inequality as in~\eqref{eq:20091702} completes the proof of~\eqref{eq:09471104}.
\end{proof}

\appendix

\section{Gradient of the Neumann Green's function} \label{app.B1}

In this section, we establish the pointwise estimate on the gradient of the Neumann Green's function which is used in the proof of Proposition~\ref{p.generalGNMcoupling22}. The statement and proof are well-known in the literature and given here for the reader's convenience.

\begin{proposition} \label{propB1}
    Fix a dimension $d \geq 3$. There exists a constant $C := C(d) < \infty$ such that the following holds. Fix $x \in \Zd$, $\ell \in \N$ and set $\Lambda = x + \Lambda_\ell$. There exists a function $\mathbf{f}_{\Lambda} : \vec{E} \left( \Lambda \right) \to \R$ satisfying the following properties: for any function $\phi : \Lambda \to \R$,
    \begin{equation} \label{eq:14310804}
        \sum_{e \in \vec{E} \left( \Lambda \right)} \mathbf{f}_{\Lambda}(e) \nabla \phi (e) = \phi(x) - \frac{1}{\left| \Lambda \right|} \sum_{x \in \Lambda} \phi(y) \hspace{5mm} \mbox{and} \hspace{5mm} \forall e \in \vec{E} \left( \Lambda\right), ~ \left| \mathbf{f}_{\Lambda}(e) \right| \leq C.
    \end{equation}
\end{proposition}

\begin{proof}
The usual version of the Sobolev inequality on $\Zd$ (which follows easily from the one on $\Rd$ by interpolation) states that there exists a constant $C := C(d) < \infty$ such that, for any function $g : \Lambda  \to \R$ satisfying $ \sum_{x \in \Lambda} g(x) = 0$,
\begin{equation*}
    \left( \sum_{y \in \Lambda} \left| g(y) \right|^{\frac{2d}{d-2}} \right)^{\frac{d-2}{2d}} \leq C \left(  \sum_{e \in \vec{E} \left( \Lambda\right)} \left| \nabla g(e)\right|^2\right)^\frac12.
\end{equation*}
In the discrete setting, we additionally have, for any function $g : \Lambda \to \R$ and any vertex $z \in \Lambda$,
\begin{equation*}
    | g(z)|  \leq \left( \sum_{y \in \Lambda} \left| g(y) \right|^{\frac{2d}{d-2}} \right)^{\frac{d-2}{2d}}.
\end{equation*}
A combination of the two previous displays yields, for any function $g : \Lambda  \to \R$ satisfying $ \sum_{y \in \Lambda} g(y) = 0$,
\begin{equation} \label{eq:14080804}
    \sup_{z \in \Lambda} \left|  g(z)\right| \leq C \left(  \sum_{e \in \vec{E} \left( \Lambda\right)} \left| \nabla g(e)\right|^2\right)^\frac12.
\end{equation}
Consider the variational problem
\begin{equation*}
    \inf_{\substack{g : \Lambda \to \R \\ \sum_{y \in \Lambda} g(y) = 0}} \frac 12 \sum_{e \in \vec{E}\left( \Lambda \right)} \left| \nabla g(e) \right|^2 - g(x).
\end{equation*}
From~\eqref{eq:14080804}, the minimum exists and is uniquely attained. We denote it by $G_{\Lambda}$ and set $\mathbf{f}_\Lambda = \nabla G_\Lambda$. By the  first variations, we have, for any function $g : \Lambda  \to \R$ satisfying $ \sum_{y \in \Lambda} g(y) = 0$,
\begin{equation} \label{eq:14300804}
    \sum_{e \in \vec{E} \left( \Lambda \right)} \mathbf{f}_{\Lambda}(e) \nabla g (e) = g(x).
\end{equation}
Consider a function $\phi : \Lambda \to \R$. Applying the identity~\eqref{eq:14300804} with the function $g := \phi - |\Lambda|^{-1}\sum_{y \in \Lambda} \phi(y) $ shows the first identity of~\eqref{eq:14310804}. We then use the inequality~\eqref{eq:14080804}, combine it with the identity~\eqref{eq:14300804} (applied with the function $g = G_{\Lambda}$) and obtain
\begin{equation} \label{eq:14400804}
    \sup_{z \in \Lambda} |G_{\Lambda}(z)| \leq C \left( \sum_{e \in \vec{E} \left( \Lambda \right)} \left| \nabla G_\Lambda (e)\right|^2 \right)^{\frac 12} = C |G_{\Lambda}(x)|^\frac 12 \leq C \left( \sup_{z \in \Lambda} |G_{\Lambda}(z)| \right)^\frac 12.
\end{equation}
A rearrangement of the inequality~\eqref{eq:14400804} yields $\sup_{z \in \Lambda} |G_{\Lambda}(z)| \leq C$. Consequently, for any edge $e = \{ y , z \} \in E \left( \Lambda \right),$ $|\mathbf{f}_{\Lambda} (e)| = \left| G_{\Lambda}(y) - G_{\Lambda}(z) \right| \leq 2C$.
\end{proof}

\section{De Giorgi-Nash-Moser regularity} \label{app.B2}

In this section, we prove the version of the De Giorgi-Nash-Moser regularity estimate for solutions of parabolic equations stated in Proposition~\ref{paraNash}. The result is again standard and well-known in the literature (see~\cite[Chapter 8]{GT},~\cite[Chapter 3]{han2011elliptic}), and we present a short proof for completeness. We mention that a closely related version of the result was established in the discrete setting considered here and used to study the $\nabla \phi$-model (without disorder) in~\cite[Appendix B]{GOS}.

We first collect two preliminary results: the mean value inequality  and the increase of oscillations for solutions of parabolic equations. In the continuous setting, these results can be found in~\cite{moser1964harnack}. In the discrete setting, we refer to Proposition 3.2 and Proposition 3.4 of~\cite{DD05}. In both statements, we fix an environment $\a : [0 , \infty] \times \Zd \to \R$ which is assumed to satisfy the uniform ellipticity assumption $c_- \leq \a \leq c_+$.

\begin{proposition}[Mean value inequality, Proposition 3.2 of~\cite{DD05}] \label{mean value inequality}
There exists a constant $C:= C(d , c_+ , c_-) < \infty$ such that, for  any $L \geq 1$, $t \geq 0$, $x \in \Zd$, and any function $u : [t , t + 4L^2] \times (x + \Lambda_{2L}) \to \R$ solution of the parabolic equation
\begin{equation*}
    \partial_t u - \nabla \cdot \a \nabla u = 0 ~\mbox{in}~[t , t + 4L^2] \times (x + \Lambda_{2L}),
\end{equation*}
one has the estimate
\begin{equation*}
    \sup_{[t + 3 L^2 , t + 4L^2] \times (x + \Lambda_L)}  u^2 \leq \frac{C}{L^{d +2}} \int_{t}^{t + 4L^2} \sum_{y \in x + \Lambda_{2L}} \left| u(t , y) \right|^2 \, dt.
\end{equation*}
\end{proposition}

In the next proposition, we use the notation $\osc$ to refer to the oscillation of a function $u$ defined by the formula
\begin{equation} \label{def.osc}
    \osc_{[s , t] \times (x + \Lambda_L)} u := \sup_{ [s, t] \times (x + \Lambda_L)} u - \inf_{[s, t] \times (x + \Lambda_L)} u.
\end{equation}

\begin{proposition}[Increase of oscillations, Proposition 3.4 of~\cite{DD05}] \label{Increase of oscillations}
There exists a constant $\theta:= \theta(d , c_+ , c_-) > 1$ such that, for any $L \geq 1$, $t \geq 0$, $x \in \Zd$ and any solution $u : [t , t+ 4L^2] \times (x + \Lambda_{2L}) \to \R$ of the parabolic equation
\begin{equation*}
    \partial_t u - \nabla \cdot \a \nabla u = 0 ~\mbox{in}~[t , t+ 4L^2] \times (x + \Lambda_{2L}),
\end{equation*}
one has the estimate
\begin{equation}
    \osc_{[t , t+ 4L^2] \times (x + \Lambda_{2L})} u \geq \theta \osc_{[t+ 3 L^2 , t+ 4L^2] \times (x + \Lambda_L)} u .
\end{equation}
\end{proposition}

\begin{proof}[Proof of Proposition~\ref{paraNash}]
The proof of Proposition~\ref{paraNash} from Proposition~\ref{mean value inequality} and Proposition~\ref{Increase of oscillations} is direct and is for instance given in~\cite[Theorem 1.31]{stroock1997markov} or in~\cite[(2.4)]{DD05}. As Proposition~\ref{paraNash} cannot be immediately obtained from the two aforementioned results, we briefly rewrite the argument.

By iterating Proposition~\ref{Increase of oscillations}, we obtain that there exists a constant $C := C(d , c_+ , c_-) < \infty$ such that, for any time $t \in [L^2/2 , L^2]$ and any edge $e = \{x , y \} \in E \left( \Lambda_{L/2}\right)$,
\begin{equation*}
    \left| \nabla u (t , e)\right| = \left| u (t , x) - u (t , y) \right| \leq  \osc_{[t - 4 , t] \times (x + \Lambda_2)} u  \leq C \theta^{-\frac{\ln L}{\ln 2}} \osc_{[t - L^2/16 , t] \times ( x + \Lambda_{L/4}) } u.
\end{equation*}
Setting $\alpha := \ln \theta /  \ln 2 > 0$, using the definition of the oscillation~\eqref{def.osc} and applying the mean value inequality stated in     Proposition~\ref{mean value inequality}, we obtain
\begin{align*}
    \left| \nabla u (t , e)\right| \leq \frac{C}{L^\alpha} \osc_{[t - L^2/16 , t] \times ( x + \Lambda_{L/4}) } u  & \leq \frac{C}{L^\alpha} \sup_{[t - L^2/16 , t] \times ( x + \Lambda_{L/4}) } |u| \\
    & \leq \frac{C}{L^\alpha} \left(  \frac{1}{L^{d +2}} \int_{t-  L^2/4}^{t} \sum_{y \in x + \Lambda_{L/2}} \left| u(s , y) \right|^2 \, ds \right)^\frac12 \\
    & \leq \frac{C}{L^\alpha} \left(  \frac{1}{L^{d +2}} \int_{0}^{L^2} \sum_{y \in \Lambda_{L}} \left| u(s , y) \right|^2 \, ds \right)^\frac12.
\end{align*}
The proof of Proposition~\ref{paraNash} is complete.
\end{proof}

\small
\bibliographystyle{abbrv}
\bibliography{RFRS}

\end{document}